\documentclass[12pt,leqn]{amsart}

\usepackage{amscd}
\usepackage{amssymb,amsfonts,amsmath,amsthm}
\usepackage{latexsym}
\usepackage{verbatim}

\usepackage{array}
\usepackage{enumerate}

\usepackage{graphicx,color}
\usepackage[colorlinks=true, pdfstartview=FitV, 
linkcolor=black, citecolor=black, urlcolor=black]{hyperref}
\usepackage{srcltx}

\numberwithin{equation}{section}

\newcommand{\CC}{\mathbb{C}}

\newcommand{\RR}{\mathbb{R}}
\newcommand{\QQ}{\mathbb{Q}}
\newcommand{\ZZ}{\mathbb{Z}}

\newcommand{\M}{\mathcal{M}}

\renewcommand{\SS}{\mathcal{S}}

\newcommand{\PP}{{\mathbb P}}

\newcommand{\CW}{\mathcal{W}}

\newcommand{\KK}{{\rm K}}
\newcommand{\Cone}{{\rm Cone}}
\newcommand{\HSm}{{\rm HS}^{\rm mon}}
\newcommand{\LL}{{\mathbb L}}
\newcommand{\height}{{\rm ht}}
\newcommand{\Var}{{\rm Var}}
\renewcommand{\(}{\left(}
\renewcommand{\)}{\right)}

\renewcommand{\dim}{{\rm dim}}

\newcommand{\Vol}{{\rm Vol}}

\newcommand{\e}{\varepsilon}
\newcommand{\Spec}{{\rm Spec}}

\newcommand{\id}{{\rm id}}
\newcommand{\supp}{{\rm supp}}

\newcommand{\Int}{{\rm Int}}
\newcommand{\relint}{{\rm rel.int}}

\newcommand{\Dbc}{{\bf D}_{c}^{b}}

\newcommand{\tl}[1]{\widetilde{#1}}

\newcommand{\simto}{\overset{\sim}{\longrightarrow}}

\newcommand{\dsum}{\displaystyle \sum}

\newcommand{\fin}{\hspace*{\fill}$\Box$\vspace*{2mm}}

\newtheorem{theorem}{Theorem}[section]

\newtheorem{corollary}[theorem]{Corollary}
\newtheorem{lemma}[theorem]{Lemma}
\newtheorem{proposition}[theorem]{Proposition}

\theoremstyle{definition}
\newtheorem{definition}[theorem]{Definition}
\theoremstyle{remark}
\newtheorem{remark}[theorem]{\sc Remark}
\newtheorem{example}[theorem]{\sc Example}
\title[Monodromies at infinity of non-tame 
polynomials]{Monodromies at infinity of non-tame 
polynomials}


\author{Kiyoshi Takeuchi}
\address{Institute of Mathematics, University  of 
Tsukuba, 1-1-1, Tennodai, 
Tsukuba, Ibaraki, 305-8571, Japan.} 
\email{takemicro@nifty.com}

\author{Mihai Tib\u ar}
\address{Math\'ematiques, Laboratoire Paul 
Painlev\'e, Universit\'e Lille 1,
59655 Villeneuve d'Ascq, France.}
\email{tibar@math.univ-lille1.fr}

\date{}

\begin{document}

\begin{abstract}
We consider the monodromy at infinity 
and the monodromies around the 
bifurcation points of polynomial functions 
$f : \CC^n \longrightarrow \CC$ which are 
not tame and might have non-isolated 
singularities. Our description of their Jordan blocks  
in terms of the Newton polyhedra 
and the motivic Milnor fibers 
relies on two new issues: 
the non-atypical eigenvalues of the 
monodromies and the corresponding 
concentration results for their 
generalized eigenspaces.
\end{abstract}

\maketitle

\section{Introduction}\label{sec:1}

For a polynomial map $f \colon \CC^n 
\longrightarrow \CC$, it is well-known that 
there exists a finite subset $B \subset \CC$ 
such that the restriction
\begin{equation}
\CC^n \setminus f^{-1}(B) \longrightarrow 
\CC \setminus B
\end{equation}
of $f$ is a locally trivial fibration. 
We denote by $B_f$ the smallest 
subset $B \subset \CC$ satisfying this 
condition.   We call the elements of $B_f$ 
\textit{bifurcation points of $f$}. 
Let $C_R=\{x\in \CC\ |\ 
|x|=R\}$ ($R \gg 0$) be a sufficiently 
large circle in $\CC$ such that 
$B_f\subset \{x \in \CC\ |\ |x|<R\}$. 
Then by restricting the locally 
trivial fibration $\CC^n \setminus 
f^{-1}(B_f) \longrightarrow \CC \setminus 
B_f$ to $C_R$ we obtain a geometric 
monodromy automorphism $\Phi_f^{\infty} 
\colon f^{-1}(R) \simto f^{-1}(R)$ and 
the linear maps
\begin{equation}
\Phi_j^{\infty} \colon H^j(f^{-1}(R) ;\CC) 
\simto 
H^j(f^{-1}(R) ;\CC) \ \ (j=0,1,\ldots)
\end{equation}
associated to it, where the 
orientation of $C_R$ is taken to be 
counter-clockwise as usual. We call 
$\Phi_j^{\infty}$'s the (cohomological) 
\textit{monodromies at infinity of $f$}. In the last 
few decades many mathematicians studied 
$\Phi_j^{\infty}$'s from various points of 
view. In particular in \cite{Broughton} Broughton 
proved that if $f$ is \emph{tame at infinity} 
(see Definition \ref{dfn:tame}) one has 
the concentration 
\begin{equation}\label{eq:conc} 
H^j(f^{-1}(R);\CC)=0 \quad (j \neq 0, n-1) 
\end{equation}
for the generic fiber $f^{-1}(R)$ 
($R\gg 0$) of $f$. 
In this case Libgober-Sperber \cite{L-S} 
obtained a beautiful formula which expresses 
the semisimple part (i.e. the eigenvalues) of 
$\Phi_{n-1}^{\infty}$ in terms of the Newton 
polyhedron at infinity of $f$ (see \cite{M-T-2} 
for its generalizations). Recently in 
\cite{M-T-4} (see also \cite{E-T}) 
the first author proved formulae for 
its nilpotent part, i.e. its Jordan normal form,  
by using the motivic Milnor fiber at infinity of $f$. 
However, the methods of the above cited 
papers do not apply beyond the tame case 
since one cannot insure 
the concentration of the cohomology  
\eqref{eq:conc} for non-tame polynomials. 
In what concerns the evaluation 
of the bifurcation set $B_f$, non-tame 
polynomials were studied
by N{\'e}methi-Zaharia \cite{N-Z}, 
Zaharia \cite{Zaharia} and many other mathematicians.

\medskip 
\par \indent 
We overcome here the above problem by 
showing that the desired cohomological concentration 
holds for the generalized eigenspaces of $\Phi_j^{\infty}$ 
for ``good" eigenvalues. Namely if we avoid some 
``bad" eigenvalues associated to $f$, we can 
successfully generalize the results in \cite{M-T-4} 
to non-tame polynomials and completely 
determine the Jordan normal forms of 
$\Phi_{n-1}^{\infty}$. In order to explain our 
results more precisely, let us recall some basic 
definitions. For $f(x)=\sum_{v\in \ZZ_+^n} 
a_vx^v$ ($a_v\in \CC$) we call the convex hull of 
$\supp f= \{ v \in \RR^n_+ \ | \ a_v \not= 0 \}$ in 
$\RR^n$ the Newton polytope of $f$ and denote it 
by $NP(f)$. After Kushnirenko \cite{Kushnirenko}, 
the convex hull $\Gamma_{\infty}(f) \subset 
\RR_+^n$ of $\{ 0 \} \cup NP(f)$ in $\RR^n$ is 
called the Newton polyhedron at infinity of $f$. 

\begin{definition}\label{COVEN}
We say that $f$ is \emph{convenient} if $\Gamma_{\infty}(f)$ 
intersects the positive part of the $i$-th axis 
of $\RR^n$ for any $1 \leq i \leq n$. 
\end{definition}
If $f$ is convenient and non-degenerate 
at infinity (for the definition, see 
Definition \ref{dfn:3-3}), 
then by a result of 
Broughton \cite{Broughton} it is tame at infinity. 
However here we do not assume that 
$f$ is convenient. In Definition 
\ref{AEV} by using the Newton polyhedron at infinity 
$\Gamma_{\infty}(f)$ we define a finite subset 
$A_f \subset \CC$ of ``bad" eigenvalues 
which we call \emph{atypical 
engenvalues} of $f$. Then the following result plays 
a key role in this paper. 
For $\lambda \in \CC$ and $j \in \ZZ$ let 
$H^{j}(f^{-1}(R);\CC)_{\lambda} \subset 
H^{j}(f^{-1}(R);\CC)$ be the 
generalized eigenspace for the eigenvalue 
$\lambda$ of the monodromy at infinity 
$\Phi_{j}^{\infty}$. 

\begin{theorem}\label{CONI} 
Let $f \in \CC [x_1, \ldots, x_n]$ be a 
non-convenient polynomial such that 
$\dim \Gamma_{\infty}(f)=n$. Assume 
that $f$ is non-degenerate at infinity. 
Then for any non-atypical eigenvalue 
$\lambda \notin A_f$ of $f$ we have 
the concentration 
\begin{equation}
H^{j}( f^{-1}(R);\CC)_{\lambda} 
\simeq 0 \qquad (j \not= n-1) 
\end{equation}
for the generic 
fiber $f^{-1}(R) \subset \CC^n$ 
($R \gg 0$) of $f$. 
\end{theorem}
This theorem allows non-isolated singularities 
of $f$ and also the situation where the fibers 
may have cohomological perturbation 
``at infinity". Indeed, some of its 
atypical fibers $f^{-1}(b)$ $(b \in B_f)$ 
e.g. $f^{-1}(0)$ have non-isolated singularities 
in general. This is the main reason why the 
monodromies of non-tame polynomials could not be studied 
successfully before. In the ``tame" case one has 
only isolated singularities in $\CC^n$ 
and either vanishing cycles 
at infinity do not occur at all or they occur at isolated points only (in the sense of \cite{S-T-1}, \cite{Tibar-book}), 
and then the concentration of cohomology 
\eqref{eq:conc} follows. 

Theorem \ref{CONI} will be proved 
by refining the proof of Sabbah's 
theorem \cite[Theorem 13.1]{Sabbah-2} in our situation. 
More precisely we construct a nice compactification 
$\tl{X_{\Sigma}}$ of $\CC^n$ and study the 
``horizontal" divisors at infinity for $f$ in 
$\tl{X_{\Sigma}} \setminus \CC^n$ very precisely to 
prove the concentration. 
With Theorem \ref{CONI} at hand, 
by using the results in \cite[Section 2]{M-T-4} 
we can easily prove the generalizations of 
\cite[Theorems 5.9, 5.14 and 5.16]{M-T-4} 
to non-tame polynomials and completely 
determine the $\lambda$-part of the Jordan 
normal form of $\Phi_{n-1}^{\infty}$ for any 
$\lambda \notin A_f$. 
Let us explain one of our results, which 
generalizes \cite[Theorem 5.9]{M-T-4}. 
Denote by $\Cone_{\infty}(f)$ 
the closed cone $\RR_+ \Gamma_{\infty}(f) 
\subset \RR^n_+$ generated by 
$\Gamma_{\infty}(f)$. 
Let $q_1,\ldots,q_l$ (resp. $\gamma_1,\ldots, 
\gamma_{l^{\prime}}$) be the 
$0$-dimensional (resp. $1$-dimensional) 
faces of $\Gamma_{\infty}(f)$ such 
that $q_i\in \Int (\Cone_{\infty}(f))$ (resp. 
the relative interior 
$\relint(\gamma_i)$ of $\gamma_i$ is 
contained in 
$\Int(\Cone_{\infty}(f))$). 
For each $q_i$ (resp. $\gamma_i$), 
denote by $d_i >0$ (resp. $e_i>0$) its 
``lattice distance" from 
the origin $0\in \RR^n$ (see Section 
\ref{sec:2} for the precise definition). 
For $1\leq i 
\leq l^{\prime}$, let $\Delta_i$ be the 
convex hull of $\{0\}\sqcup 
\gamma_i$ in $\RR^n$. Then for $\lambda 
 \not= 1$ and $1 \leq 
i \leq l^{\prime}$ such that $\lambda^{e_i}=1$ we set
\begin{equation}
n(\lambda)_i
= \sharp\{ v\in \ZZ^n \cap \relint(\Delta_i) 
\ |\ \height (v, \gamma_i)=k\} 
+\sharp \{ v\in \ZZ^n \cap \relint(\Delta_i) 
\ |\ \height (v, 
\gamma_i)=e_i-k\},
\end{equation}
where $k$ is the minimal positive 
integer satisfying 
$\lambda=\zeta_{e_i}^{k}$ 
(we set $\zeta_{d}:=\exp 
(2\pi\sqrt{-1}/d) \in \CC$) 
and for $v\in \ZZ^n \cap \relint(\Delta_i)$ we 
denote by $\height (v, \gamma_i)$ the 
lattice height of $v$ from the base 
$\gamma_i$ of $\Delta_i$. 
Then we have the following extension of 
\cite[Theorem 5.9]{M-T-4} from tame 
to non-tame polynomials. 

\begin{theorem}
Assume that 
$\dim \Gamma_{\infty}(f)=n$ and 
$f$ is non-degenerate at infinity. Then for any 
$\lambda \notin A_f$ we have
\begin{enumerate}
\item The number of the Jordan blocks 
for the eigenvalue $\lambda$ with the 
maximal possible size $n$ in 
$\Phi_{n-1}^{\infty} \colon 
H^{n-1}(f^{-1}(R);\CC) \simto H^{n-1}(f^{-1}(R);\CC)$ 
($R \gg 0$) is equal 
to $\sharp \{q_i \ |\ \lambda^{d_i}=1\}$.
\item The number of the Jordan blocks for 
the eigenvalue $\lambda$ with the second maximal possible size 
$n-1$ in $\Phi_{n-1}^{\infty}$ is 
equal to $\sum_{i \colon \lambda^{e_i}=1} 
n(\lambda)_i$.
\end{enumerate}
\end{theorem}

\medskip \par 
\indent We can treat in a similar manner 
the monodromies at bifurcation points of
$f$. Let $b \in B_f$ be such a bifurcation 
point. Choose sufficiently 
small $\e >0$ such that 
\begin{equation}
B_f \cap \{ x \in \CC \ | \ |x-b| \leq \e \} 
= \{ b \} 
\end{equation}
and set $C_{\e}(b)= 
\{ x \in \CC \ | \ |x-b|= \e \} \subset \CC$. 
Then we obtain a locally trivial fibration 
$f^{-1}(C_{\e}(b)) \longrightarrow C_{\e}(b)$ 
over the small circle $C_{\e}(b) \subset \CC$ 
and the monodromy automorphisms 
\begin{equation}
\Phi_j^b \colon H^j(f^{-1}(b+ \e ) ;\CC) 
\simto 
H^j(f^{-1}(b+ \e ) ;\CC) \ \ (j=0,1,\ldots)
\end{equation}
around the atypical fiber $f^{-1}(b) \subset \CC^n$ 
associated to it. In Section \ref{sec:6} we apply 
our methods to the Jordan normal forms of 
$\Phi_j^b$'s. If $f$ is 
not convenient, then for some $b \in B_f$ the atypical 
fiber $f^{-1}(b) \subset \CC^n$ may have 
``singularities at infinity". Even in such 
cases, we can define a finite subset 
$A_{f,b}^{\circ} \subset \CC$ of ``bad" eigenvalues 
for $b \in B_f$ and completely determine 
the $\lambda$-part of the Jordan 
normal form of $\Phi_{n-1}^{b}$ for any 
$\lambda \notin A_{f,b}^{\circ}$. 
In fact we obtain these results more generally, 
for polynomial maps $f: U \longrightarrow \CC$ 
of affine algebraic varieties $U$. 
See Section \ref{sec:6} for the details.

\section{Preliminaries}\label{sec:2}

Let $f \colon \CC^n 
\longrightarrow \CC$ be the polynomial map 
in Section \ref{sec:1}. 
To study its monodromies at infinity 
$\Phi_j^{\infty}$, we often impose the 
following natural condition.

\begin{definition}[\cite{Kushnirenko}]\label{dfn:tame}
Let $\partial f\colon \CC^n 
\longrightarrow \CC^n$ be the map defined by 
$\partial f(x)=(\partial_1f(x), 
\ldots, \partial_n f(x))$. Then we say that 
$f$ is \emph{tame at infinity} if the 
restriction $(\partial f)^{-1}(B(0;\e )) 
\longrightarrow B(0;\e )$ of $\partial f$ 
to a sufficiently small ball 
$B(0;\e )$ centered at the origin 
$0 \in \CC^n$ is proper.
\end{definition}

The following result is fundamental in 
the study of monodromies at infinity.

\begin{theorem}[Broughton 
\cite{Broughton} and Siersma-Tib{\u a}r 
\cite{S-T-1}]\label{tame}
Assume that $f$ is tame at infinity (more generally, 
with isolated $\CW$-singularities, see \cite{S-T-1}). 
Then the generic fiber $f^{-1}(c)$ ($c 
\in \CC \setminus B_f$) of $f$ 
has the homotopy type of the bouquet of 
$(n-1)$-spheres. In particular, we have
\begin{equation}
H^j(f^{-1}(c);\CC)=0 \quad (j \neq 0, n-1).
\end{equation} \fin
\end{theorem}
If $f$ is tame at 
infinity, then of course $\Phi_{n-1}^{\infty}$ is 
the unique non-trivial monodromy at infinity. 

\begin{definition}[\cite{Kushnirenko}]\label{dfn:3-3}
We say that the polynomial $f(x)=\sum_{v\in \ZZ_+^n} 
a_vx^v$ ($a_v\in \CC$) is 
\emph{non-degenerate at infinity} if for any 
face $\gamma$ of $\Gamma_{\infty}(f)$ 
such that $0 \notin \gamma$ the complex 
hypersurface $\{x \in (\CC^*)^n\ |\ 
f_{\gamma}(x)=0\}$ in $(\CC^*)^n$ is 
smooth and reduced, where we set 
$f_{\gamma}(x)=\sum_{v \in \gamma 
\cap \ZZ_+^n} a_vx^v$.
\end{definition}

If $f$ is convenient and non-degenerate 
at infinity, then by a result of 
Broughton \cite{Broughton} it is tame at infinity. 
However in this paper, we do not assume that 
$f$ is convenient. 

\begin{definition}
Assume that $\dim \Gamma_{\infty}(f)=n$. Then 
we say that a face $\gamma \prec \Gamma_{\infty}(f)$ 
is \emph{atypical} if $0 \in \gamma$ and 
there exists a facet i.e. an 
$(n-1)$-dimensional face $\Gamma$ of 
$\Gamma_{\infty}(f)$ containing $\gamma$ 
whose non-zero inner conormal 
vectors are not contained in the 
first quadrant $\RR^n_+$ of $\RR^n$. 
\end{definition}

\begin{remark}
Our definition above is closely related to that of 
\emph{bad faces} of $NP(f)$ in N{\'e}methi-Zaharia \cite{N-Z}. 
If $\gamma \prec NP(f)$ is a bad face of $NP(f)$, then 
the convex hull of $\{ 0 \} \cup \gamma$ in $\RR^n$ 
is an atypical one of $\Gamma_{\infty}(f)$. However, not
all the atypical faces of $\Gamma_{\infty}(f)$ 
are obtained in this way. 
\end{remark} 

\begin{example}
Let $n=3$ and consider a non-convenient polynomial 
$f(x,y,z)$ on $\CC^3$ whose Newton polyhedron at 
infinity $\Gamma_{\infty}(f)$ is the convex hull of 
the points $(2,0,0), (2,2,0), (2,2,3) \in \RR^3_+$ 
and the origin $0=(0,0,0) \in \RR^3$. Then the line 
segment connecting the point $(2,2,0)$ 
and the origin $0 \in \RR^3$ is an atypical 
face of $\Gamma_{\infty}(f)$.  However the 
triangle whose vertices are the points 
$(2,0,0)$, $(2,2,0)$ 
and the origin $0 \in \RR^3$ is not so. 
\end{example}

\begin{example}
Let $n=3$ and consider a non-convenient polynomial 
$f(x,y,z)$ on $\CC^3$ whose Newton polyhedron at 
infinity $\Gamma_{\infty}(f)$ is the convex hull of 
the points $(2,0,0), (0,2,0), (1,1,2) \in \RR^3_+$ 
and the origin $0=(0,0,0) \in \RR^3$. Then the line 
segment connecting the point $(2,0,0)$ 
and the origin $0 \in \RR^3$ is an atypical 
face of $\Gamma_{\infty}(f)$. 
\end{example}
If $\dim \Gamma_{\infty}(f)=n$, to the $n$-dimensional 
integral polytope $\Gamma_{\infty}(f)$ in $\RR^n$ 
we can naturally associate a subdivision of 
(the dual vector space of) $\RR^n$ into 
rational convex polyhedral cones as follows. For an element 
$u \in \RR^n$ of 
(the dual vector space of) $\RR^n$ define the 
supporting face $\gamma_u \prec  \Gamma_{\infty}(f)$ 
of $u$ in $ \Gamma_{\infty}(f)$ by 
\begin{equation}
\gamma_u = \left\{ v \in \Gamma_{\infty}(f) \ | \ 
\langle u , v \rangle 
= 
\min_{w \in \Gamma_{\infty}(f)} 
\langle u ,w \rangle \right\}. 
\end{equation}
Then we introduce an equivalence relation $\sim$ on 
(the dual vector space of) $\RR^n$ by 
$u \sim u^{\prime} \Longleftrightarrow 
\gamma_u = \gamma_{u^{\prime}}$. We can easily 
see that for any face $\gamma \prec  \Gamma_{\infty}(f)$ 
of $\Gamma_{\infty}(f)$ the closure of the 
equivalence class associated to $\gamma$ in $\RR^n$ 
is an $(n- \dim \gamma )$-dimensional rational 
convex polyhedral cone $\sigma (\gamma)$ in $\RR^n$. Moreover 
the family $\{ \sigma (\gamma) \ | \ 
\gamma \prec  \Gamma_{\infty}(f) \}$ of cones in $\RR^n$ 
thus obtained is a subdivision of $\RR^n$ and 
satisfies the axiom of fans (see \cite{Fulton} and 
\cite{Oda} etc.). We call it the dual fan of 
$\Gamma_{\infty}(f)$. Then we have the following 
characterization of atypical 
faces of $\Gamma_{\infty}(f)$. 

\begin{lemma} 
Assume that $\dim \Gamma_{\infty}(f)=n$ and let 
$\gamma \prec \Gamma_{\infty}(f)$ be a face of 
$\Gamma_{\infty}(f)$ such that $0 \in \gamma$. 
Then $\gamma$ is atypical if and only if 
the cone $\sigma ( \gamma )$ 
which corresponds to it in the dual 
fan of $\Gamma_{\infty}(f)$ is not contained in 
$\RR^n_+$. 
\fin
\end{lemma}
For a subset $S \subset \{ 1,2, \ldots, n \}$ 
we define a coordinate subspace $\RR^S \simeq 
\RR^{|S|}$ of $\RR^n$ by 
\begin{equation}
\RR^S = \{ v=(v_1, \ldots, v_n) \in \RR^n \ | \ 
v_i=0 \quad \text{for any} \quad i \notin S \}. 
\end{equation}
The following lemma should be obvious. 

\begin{lemma} 
Assume that $\dim \Gamma_{\infty}(f)=n$ and let 
$\gamma \prec \Gamma_{\infty}(f)$ be a face of 
$\Gamma_{\infty}(f)$ such that $0 \in \gamma$. 
Let $\RR^S \subset \RR^n$ be the minimal 
coordinate subspace of $\RR^n$ containing 
$\gamma$ and assume that $\dim \gamma < \dim 
\RR^S = |S|$. Then $\gamma$ is an atypical 
face of $\Gamma_{\infty}(f)$. 
\fin
\end{lemma}
By this lemma we can easily prove the following 
proposition. 

\begin{proposition}\label{ADD-1} 
Assume that $\dim \Gamma_{\infty}(f)=n$ and a face 
$\gamma \prec \Gamma_{\infty}(f)$ of 
$\Gamma_{\infty}(f)$ such that $0 \in \gamma$ 
is not atypical. Let $\RR^S \subset \RR^n$ be the minimal 
coordinate subspace of $\RR^n$ containing 
$\gamma$. Then we have $\dim \gamma = 
\dim \RR^S = |S|$ and there exist exactly 
$n - |S|$ facets i.e. $(n-1)$-dimensional faces 
$\Gamma_i \prec \Gamma_{\infty}(f)$ ($i \notin S$) of 
$\Gamma_{\infty}(f)$ containing $\gamma$. 
Moreover they are explicitly given by 
\begin{equation}
\Gamma_i = \Gamma_{\infty}(f) \cap 
\{ v=(v_1, \ldots, v_n) \in \RR^n \ | \ v_i=0 \} 
\qquad (i \notin S). 
\end{equation} \fin
\end{proposition} 

\begin{definition} 
We say that a face $\gamma 
\prec \Gamma_{\infty}(f)$ of $\Gamma_{\infty}(f)$ 
is \emph{at infinity} if $0 \notin \gamma$. 
We say that such a face $\gamma$ is moreover 
\emph{admissible} if it is not contained in any atypical 
face of $\Gamma_{\infty}(f)$. 
\end{definition}

For a face at infinity 
$\gamma \prec 
\Gamma_{\infty}(f)$ of $\Gamma_{\infty}(f)$, 
let $\Delta_{\gamma}$ be the convex hull of 
$\{0\} \sqcup \gamma$ in $\RR^n$. Denote by 
$\LL(\Delta_{\gamma})$ the $(\dim 
\gamma +1)$-dimensional linear subspace 
of $\RR^n$ spanned by $\Delta_{\gamma}$ 
and consider the lattice 
$M_{\gamma}=\ZZ^n \cap \LL(\Delta_{\gamma}) 
\simeq \ZZ^{\dim \gamma+1}$ in it. 
Then there exists a unique non-zero primitive 
vector $u_{\gamma}$ in its dual lattice 
which takes its maximum in $\Delta_{\gamma}$ 
exactly on $\gamma \prec \Delta_{\gamma}$: 
\begin{equation}
\gamma = \left\{ v \in \Delta_{\gamma} \ | \ 
\langle u_{\gamma} ,v \rangle 
= 
\max_{w \in \Delta_{\gamma}} 
\langle u_{\gamma} ,w \rangle \right\}. 
\end{equation}
We set 
\begin{equation}
d_{\gamma}
= \max_{w \in \Delta_{\gamma}} 
\langle u_{\gamma} ,w \rangle \in \ZZ_{>0} 
\end{equation}
and call it the 
lattice distance of $\gamma$ from the 
origin $0 \in \RR^n$. The following definition 
will be used in Sections \ref{sec:3} and \ref{sec:5}. 

\begin{definition}\label{AEV} 
Assume that $\dim \Gamma_{\infty}(f)=n$. 
Then we say that a complex number $\lambda \in \CC$ 
is an \emph{atypical eigenvalue of $f$} if either 
$\lambda =1$ or there 
exists a non-admissible face at infinity $\gamma 
\prec \Gamma_{\infty}(f)$ of 
$\Gamma_{\infty}(f)$ such that 
$\lambda^{d_{\gamma}}=1$. We denote by $A_f 
\subset \CC$ the set of the atypical 
eigenvalues of $f$. 
\end{definition}

\begin{example}
Let $n=2$ and consider a non-convenient polynomial 
$f(x,y)$ on $\CC^2$ whose Newton polyhedron at 
infinity $\Gamma_{\infty}(f)$ is the convex hull of 
the points $(1,3), (3,0), (3,2) \in \RR^2_+$ 
and the origin $0=(0,0) \in \RR^2$. 
Then the line segment connecting the point $(1,3)$ and 
the origin is the unique atypical face of 
$\Gamma_{\infty}(f)$ and we have $A_f = \{ 1 \}$. 
\end{example}
Let $\overrightarrow{e_i} 
= ^t(0, \ldots, 0, \overset{i}{1}, 0, \ldots, 0) 
\in \RR^n$ ($i=1,2, \ldots, n$) be 
the standard basis of 
$\RR^n$. The following result, whose proof 
immediately follows from Proposition \ref{ADD-1}, 
will be used in Section \ref{sec:3}.

\begin{proposition}\label{ADD-2} 
Assume that $\dim \Gamma_{\infty}(f)=n$ and let 
$\gamma \prec \Gamma_{\infty}(f)$ 
be a face at infinity of 
$\Gamma_{\infty}(f)$. 
Let $\RR^S \subset \RR^n$ be the minimal 
coordinate subspace of $\RR^n$ containing 
$\gamma$ and $\sigma \subset \RR^n$ 
the cone in the dual fan of $\Gamma_{\infty}(f)$ 
which corresponds to 
$\gamma \prec \Gamma_{\infty}(f)$. Then 
$\gamma$ is admissible if and only if 
there exist some integral vectors 
$\overrightarrow{f_1}, \ldots, \overrightarrow{f_k} 
\in (\RR^n \setminus \RR^n_+) \cap \ZZ^n$ such that 
\begin{equation}
\min_{v \in \Gamma_{\infty}(f)} 
\langle \overrightarrow{f_j} ,v \rangle < 0 
\qquad (1 \leq j \leq k) 
\end{equation}
and 
\begin{equation}
\sigma = ( \sum_{i \notin S} \RR_+ 
\overrightarrow{e_i} ) 
+  ( \sum_{j=1}^k \RR_+ 
\overrightarrow{f_j} ) . 
\end{equation}\fin
\end{proposition}

\section{Motivic Milnor fibers at infinity}\label{sec:3}

From now, following Denef-Loeser 
\cite{D-L-1}, \cite{D-L-2} and 
Guibert-Loeser-Merle \cite{G-L-M} we 
introduce motivic reincarnations 
of global (Milnor) fibers of polynomial 
maps. For the details, see 
Matsui-Takeuchi \cite{M-T-4}, 
Esterov-Takeuchi \cite{E-T} and 
Raibaut \cite{Raibaut}. 
We also follow the terminologies 
in \cite{Dimca}, \cite{H-T-T} and \cite{K-S} etc. 
Let $f: \CC^n \longrightarrow \CC$ be a 
general polynomial map. 
First, take a smooth compactification 
$X$ of $\CC^n$. 
Next, by eliminating the 
points of indeterminacy of the 
meromorphic extension of $f$ to $X$ we obtain 
a commutative diagram

\begin{equation}
\begin{CD}
\CC^n  @>{\iota}>> \tl{X}
\\
@V{f}VV   @VV{g}V
\\
\CC @>>{j}> \PP^1 
\end{CD}
\end{equation}
such that horizontal arrows are open 
embeddings, $g$ is a proper 
holomorphic map and $\tl{X} \setminus \CC^n$, 
$Y:=g^{-1}( \infty )$ are 
normal crossing divisors in $\tl{X}$. 
Take a local coordinate $h$ of $\PP^1$ in a 
neighborhood of $\infty\in \PP^1$ such that 
$\infty=\{h=0\}$ and set 
$\tl{g}=h\circ g$. Note that $\tl{g}$ is a 
holomorphic function defined on 
a neighborhood of the closed subvariety 
$Y=\tl{g}^{-1}(0)=g^{-1}(\infty) 
\subset \tl{X} \setminus \CC^n$ of 
$\tl{X}$. Then for $R \gg 0$ we have
\begin{equation}\label{eq:4-2}
H_c^j(f^{-1}(R);\CC) \simeq 
H^j \psi_h (j_! Rf_! \CC_{\CC^n}) \simeq 
H^j \psi_h ( Rg_! \iota_! \CC_{\CC^n}) \simeq 
H^j(Y; \psi_{\tl{g}}(\iota_! \CC_{\CC^n})), 
\end{equation}
where $\psi_h$ and $\psi_{\tl{g}}$ are 
nearby cycle functors 
(for the definition, see \cite{Dimca} and \cite{K-S} 
etc.). Let us define an open 
subset $\Omega$ of $\tl{X}$ by
\begin{equation}
\Omega=\Int (\iota(\CC^n) \sqcup Y)
\end{equation}
and set $U=\Omega \cap Y$. 
Then $U$ (resp. the complement of $\Omega$ in 
$\tl{X}$) is a normal crossing divisor 
in $\Omega$ (resp. $\tl{X}$). By using this 
very special geometric situation 
we can easily prove the isomorphisms
\begin{equation}\label{eq:4-5} 
H^j(Y;\psi_{\tl{g}}(\iota_! \CC_{\CC^n}))
\simeq H^j(Y; 
\psi_{\tl{g}}(\iota_!^{\prime}
\CC_{\Omega}))\simeq H_c^j(U; 
\psi_{\tl{g}}(\CC_{\tl{X}})),
\end{equation}
where $\iota^{\prime} \colon \Omega 
\hookrightarrow \tl{X}$ is the 
inclusion. Now let $E_1, E_2, \ldots, E_k$ 
be the irreducible components of 
the normal crossing divisor $U=\Omega 
\cap Y$ in $\Omega \subset \tl{X}$. 
For each $1 \leq i \leq k$, let $b_i>0$ be 
the order of the zero of $\tl{g}$ 
along $E_i$. For a non-empty subset $I 
\subset \{1,2,\ldots, k\}$, let us 
set
\begin{equation}
E_I=\bigcap_{i \in I} E_i,\hspace{10mm}
E_I^{\circ}=E_I \setminus \bigcup_{i 
\not\in I}E_i
\end{equation}
and $d_I= {\rm gcd} (b_i)_{i \in I}>0$. Then, 
as in \cite[Section 3.3]{D-L-2}, we 
can construct an unramified Galois 
covering $\tl{E_I^{\circ}} 
\longrightarrow E_I^{\circ}$ of $E_I^{\circ}$ 
as follows. First, for a point 
$p \in E_I^{\circ}$ we take an affine open 
neighborhood $W \subset \Omega 
\setminus ( \cup_{i \notin I} E_i)$ of $p$ 
on which there exist regular functions 
$\xi_i$ $(i \in I)$ 
such that $E_i \cap W=\{ \xi_i=0 \}$ 
for any $i \in I$. Then on $W$ we have 
$\tl{g}=\tl{g_{1,W}} (\tl{g_{2,W}})^{d_I}$, 
where we set 
$\tl{g_{1,W}}=\tl{g} \prod_{i \in I}\xi_i^{-b_i}$ 
and $\tl{g_{2,W}}=\prod_{i 
\in I} \xi_i^{\frac{b_i}{d_I}}$. 
Note that $\tl{g_{1,W}}$ is a unit on $W$ 
and $\tl{g_{2,W}} \colon W \longrightarrow 
\CC$ is a regular function. It is 
easy to see that $E_I^{\circ}$ is covered 
by such affine open subsets $W$ of 
$\Omega \setminus ( \cup_{i \notin I} E_i)$. 
Then as in \cite[Section 3.3]{D-L-2} 
by gluing the varieties
\begin{equation}\label{eq:6-26}
\tl{E_{I,W}^{\circ}}=\{(t,z) \in 
\CC^* \times (E_I^{\circ} \cap W) \ |\ 
t^{d_I} =(\tl{g_{1,W}})^{-1}(z)\}
\end{equation}
together in the following way, 
we obtain the variety $\tl{E_I^{\circ}}$ over 
$E_I^{\circ}$. If $W^{\prime}$ 
is another such open subset and 
$\tl{g}=\tl{g_{1,W^{\prime}}} 
(\tl{g_{2,W^{\prime}}})^{d_I}$ is the 
decomposition of $\tl{g}$ on it, 
we patch $\tl{E_{I,W}^{\circ}}$ and 
$\tl{E_{I,W^{\prime}}^{\circ}}$ by 
the morphism $(t,z) \longmapsto 
(\tl{g_{2,W^{\prime}}}(z)( \tl{g_{2,W}})^{-1}(z) 
\cdot t, z)$ defined over 
$W \cap W^{\prime}$. 
Now for $d \in \ZZ_{>0}$, let 
$\mu_d \simeq \ZZ/\ZZ d$ be the multiplicative 
group consisting of the $d$-roots 
in $\CC$. We denote by $\hat{\mu}$ the 
projective limit $\underset{d}{\varprojlim} 
\mu_d$ of the projective system 
$\{ \mu_i \}_{i \geq 1}$ with morphisms 
$\mu_{id} \longrightarrow \mu_i$ 
given by $t \longmapsto t^d$. Then the 
unramified Galois covering 
$\tl{E_I^{\circ}}$ of $E_I^{\circ}$ 
admits a natural $\mu_{d_I}$-action 
defined by assigning the automorphism 
$(t,z) \longmapsto (\zeta_{d_I} t, z)$ 
of $\tl{E_I^{\circ}}$ to the generator 
$\zeta_{d_I}:=\exp 
(2\pi\sqrt{-1}/d_I) \in \mu_{d_I}$. Namely 
the variety $\tl{E_I^{\circ}}$ is 
equipped with a good $\hat{\mu}$-action 
in the sense of \cite[Section 
2.4]{D-L-2}. Following the notations 
in \cite{D-L-2}, denote by 
$\M_{\CC}^{\hat{\mu}}$ the ring obtained 
from the Grothendieck ring 
$\KK_0^{\hat{\mu}}(\Var_{\CC})$ of 
varieties over $\CC$ with good 
$\hat{\mu}$-actions by inverting the 
Lefschetz motive $\LL\simeq \CC \in 
\KK_0^{\hat{\mu}}(\Var_{\CC})$. Recall 
that $\LL \in 
\KK_0^{\hat{\mu}}(\Var_{\CC})$ is endowed 
with the trivial action of 
$\hat{\mu}$.

\begin{definition}[\cite{M-T-4} and 
\cite{Raibaut}] We define the \emph{motivic 
Milnor fiber at 
infinity} $\SS_f^{\infty}$ of the 
polynomial map $f \colon \CC^n 
\longrightarrow \CC$ by
\begin{equation}\label{MMF}
\SS_f^{\infty} =\sum_{I \neq \emptyset} 
(1-\LL)^{|I| -1} 
[\tl{E_I^{\circ}}] \in 
\M_{\CC}^{\hat{\mu}}.
\end{equation}
\end{definition}

\begin{remark}
By Guibert-Loeser-Merle \cite[Theorem 3.9]{G-L-M}, 
the motivic Milnor fiber 
at infinity $\SS_f^{\infty}$ of $f$ does 
not depend on the compactification 
$X$ of $\CC^n$. This fact was informed to 
us by Sch{\"u}rmann (private 
communication) and Raibaut \cite{Raibaut}.
\end{remark}

As in \cite[Section 3.1.2 and 3.1.3]{D-L-2}, 
we denote by $\HSm$ the abelian 
category of Hodge structures with a 
quasi-unipotent endomorphism. Then, to 
the object $\psi_h(j_!Rf_!\CC_{\CC^n})\in 
\Dbc(\{\infty\})$ and the 
semisimple part of the monodromy 
automorphism acting on it, we can associate 
an element
\begin{equation}
[H_f^{\infty}] \in \KK_0(\HSm)
\end{equation}
as in \cite{D-L-1} and \cite{D-L-2}. 
Recall that the weight filtrations of 
$H^j \psi_h(j_!Rf_!\CC_{\CC^n})$ 
in the construction of $[H_f^{\infty}]$ 
are ``relative" monodromy filtrations. 
To describe the element 
$[H_f^{\infty}]\in \KK_0(\HSm)$ in terms of 
$\SS_f^{\infty}\in \M_{\CC}^{\hat{\mu}}$, let
\begin{equation}
\chi_h \colon \M_{\CC}^{\hat{\mu}} 
\longrightarrow \KK_0(\HSm)
\end{equation}
be the Hodge characteristic morphism 
defined in \cite{D-L-2} which 
associates to a variety $Z$ with a 
good $\mu_d$-action the Hodge structure
\begin{equation}
\chi_h ([Z])=\sum_{j \in \ZZ} (-1)^j 
[H_c^j(Z;\QQ)] \in \KK_0(\HSm)
\end{equation}
with the actions induced by the one 
$z \longmapsto \exp (2\pi\sqrt{-1}/d)z$ 
($z\in Z$) on $Z$. Then 
as in \cite[Theorem 4.4]{M-T-4}, by applying 
the proof of \cite[Theorem 4.2.1]{D-L-1} 
to our situation \eqref{eq:4-2} and 
\eqref{eq:4-5}, we obtain the following 
result.

\begin{theorem}\label{thm:7-6}
In the Grothendieck group $\KK_0(\HSm)$, we have 
the equality 
\begin{equation}
[H_f^{\infty}]=\chi_h(\SS_f^{\infty}).
\end{equation}\fin
\end{theorem}

By using Newton polyhedrons at 
infinity, we can rewrite 
Theorem \ref{thm:7-6} more explicitly 
as follows. Let $f\in \CC[x_1,\ldots,x_n]$ 
be a ``non-convenient" polynomial such that 
$\dim \Gamma_{\infty}(f)=n$. 
Assume that $f$ is non-degenerate at infinity. 
Now let us consider $\CC^n$ as a 
toric variety associated with the fan 
$\Sigma_0$ in $\RR^n$ formed by all the 
faces of the first quadrant 
$\RR_+^n \subset \RR^n$. 
Denote by $T \simeq (\CC^*)^n$ the 
open dense torus in it. Let $\Sigma_1$ 
be a subdivision of the dual fan 
of $\Gamma_{\infty}(f)$ which contains 
$\Sigma_0$ as its subfan. Then we can 
construct a smooth subdivision $\Sigma$ of 
$\Sigma_1$ without subdividing 
the cones in $\Sigma_0$ (see e.g. 
\cite[Lemma (2.6), Chapter II, page 
99]{Oka}). 
This implies that the 
toric variety $X_{\Sigma}$ associated with 
$\Sigma$ is a smooth compactification 
of $\CC^n$. Our construction of $X_{\Sigma}$ 
coincides with the one in Zaharia \cite{Zaharia}. 
Recall that $T$ acts on 
$X_{\Sigma}$ and the $T$-orbits are 
parametrized by the cones in $\Sigma$. 
For a cone $\sigma \in \Sigma$ denote 
by $T_{\sigma} \simeq (\CC^*)^{n-\dim 
\sigma}$ the corresponding $T$-orbit. 
We have also natural affine open 
subsets $\CC^n(\sigma) \simeq \CC^n$ of 
$X_{\Sigma}$ associated to 
$n$-dimensional cones $\sigma$ in $\Sigma$ 
as follows. Let $\sigma$ be an 
$n$-dimensional cone in $\Sigma$ and 
$\{w_1,\ldots, w_n\} \subset \ZZ^n$ the 
set of the (non-zero) 
primitive vectors on the edges of 
$\sigma$. Then by the smoothness of $\Sigma$ 
the semigroup ring $\CC [ \ZZ^n \cap \sigma ]$ 
is isomorphic to the polynomial ring 
$\CC [y_1, \ldots, y_n]$. 
This implies that the affine open 
subset $\CC^n(\sigma):= 
\Spec ( \CC [ \ZZ^n \cap \sigma ])$ 
of $X_{\Sigma}$ is isomorphic to $\CC^n_y$. 
Moreover, on $\CC^n(\sigma) \simeq 
\CC^n_y$ the function $f$ has the 
following form:
\begin{equation}
f(y)=\sum_{v \in \ZZ_+^n}
a_{v}y_1^{\langle w_1,v \rangle}\cdots y_n^{\langle 
w_n, v \rangle} =y_1^{b_1} \cdots y_n^{b_n} 
\times f_{\sigma}(y),
\end{equation}
where we set $f=\sum_{v \in \ZZ_+^n}a_{v}x^{v}$,
\begin{equation}
b_i=\min_{v\in \Gamma_{\infty}(f)} 
\langle w_i,v \rangle \leq 0 \qquad 
(i=1,2,\ldots,n)
\end{equation}
and $f_{\sigma}(y)$ is a polynomial on 
$\CC^n(\sigma) \simeq \CC^n_y$. In 
$\CC^n(\sigma) \simeq \CC^n_y$ 
the hypersurface $Z:= \overline{f^{-1}(0)} 
\subset X_{\Sigma}$ is explicitly 
written as $\{ y \in \CC^n(\sigma) \ | \ 
f_{\sigma}(y)=0 \}$. 
The variety $X_{\Sigma}$ is covered by 
such affine open 
subsets. Let $\tau$ be a 
$d$-dimensional face of the $n$-dimensional cone 
$\sigma \in \Sigma$. For simplicity, 
assume that $w_1,\ldots, w_d$ generate 
$\tau$. Then in the affine chart 
$\CC^n(\sigma) \simeq \CC^n_y$ the 
$T$-orbit $T_{\tau}$ associated to 
$\tau$ is explicitly defined by
\begin{equation*}
T_{\tau}=\{(y_1,\ldots,y_n)\in 
\CC^n(\sigma)\  |\ y_1=\cdots =y_d=0,\ 
y_{d+1},\ldots, y_n\neq 0\}\simeq (\CC^*)^{n-d}.
\end{equation*}
Hence we have
\begin{equation}
 X_{\Sigma}=\bigcup_{\dim \sigma =n}
\CC^n(\sigma)=\bigsqcup_{\tau 
\in\Sigma}T_{\tau}.
\end{equation}
Now $f$ extends to a meromorphic 
function on $X_{\Sigma}$, 
which may still have points of indeterminacy. 
For simplicity we denote this 
meromorphic extension also by $f$. 
From now on, we will eliminate 
its points of indeterminacy 
by blowing up $X_{\Sigma}$ 
(see \cite[Section 3]{M-T-2} 
and \cite[Section 3]{M-T-4} for 
the details). 
For a cone $\sigma$ in $\Sigma$ by 
taking a non-zero vector $u$ in the 
relative interior $\relint(\sigma)$ of 
$\sigma$ we define a face $\gamma(\sigma)$ 
of $\Gamma_{\infty}(f)$ by
\begin{equation}
\gamma(\sigma) =\left\{ v 
\in \Gamma_{\infty}(f) \ | \ \langle u ,v 
\rangle = \min_{w \in \Gamma_{\infty}(f)} 
\langle u,w \rangle \right\}.
\end{equation}
This face $\gamma(\sigma)$ does 
not depend on the choice of $u \in 
\relint(\sigma)$ and is called 
the supporting face of $\sigma$ in 
$\Gamma_{\infty}(f)$. Following \cite{L-S}, 
we say that a $T$-orbit 
$T_{\sigma}$ in $X_{\Sigma}$ 
(or a cone $\sigma \in \Sigma$) is at infinity 
if its supporting face $\gamma(\sigma) \prec 
\Gamma_{\infty}(f)$ is at infinity i.e. 
$0 \notin \gamma(\sigma)$. 
We can easily see that $f$ has poles on 
the union of $T$-orbits at infinity. 
Let $\rho_1, \rho_2, \ldots, \rho_m$ be 
the $1$-dimensional cones 
at infinity in $\Sigma$ 
and set $T_i=T_{\rho_i}$. We call the 
cones $\rho_i$ rays at infinity in $\Sigma$. 
Then $T_1,T_2,\ldots, T_m$ are the 
$(n-1)$-dimensional $T$-orbits at infinity 
in $X_{\Sigma}$. For any 
$i=1,2,\ldots,m$ the toric divisor 
$D_i:=\overline{T_i}$ is a smooth 
hypersurface in $X_{\Sigma}$ and the 
poles of $f$ are contained in $D_1 
\cup \cdots \cup D_m$. 
Let us denote the (unique non-zero) 
primitive vector in $\rho_i \cap 
\ZZ^n$ by $u_i$. Then the order $a_i>0$ 
of the pole of $f$ along $D_i$ 
is given by
\begin{equation}
a_i=-\min_{v\in \Gamma_{\infty}(f)} 
\langle u_i,v \rangle.
\end{equation}
Moreover by the non-convenience of $f$, there 
exist some cones $\sigma \in \Sigma$ such that 
$\sigma \notin \Sigma_0$ and $0 \in \gamma (\sigma )$ 
i.e. $\gamma (\sigma )$ is an atypical face of 
$\Gamma_{\infty}(f)$. For such $\sigma$ the 
function $f$ extends holomorphically to a 
neighborhood of $T_{\sigma} \subset 
X_{\Sigma} \setminus \CC^n$. For this 
reason we call them ``horizontal" 
$T$-orbits in $X_{\Sigma}$. 
Note also that by the 
non-degeneracy at infinity of $f$, for 
any non-empty subset $I \subset \{ 
1,2, \ldots, m \}$ 
the hypersurface $Z= \overline{f^{-1}(0)}$ 
in $X_{\Sigma}$ intersects $D_I:= 
\bigcap_{i \in I}D_i$ transversally 
(or the intersection is empty). At such intersection 
points, $f$ has indeterminacy. 
Now, in order to eliminate the 
indeterminacy of the meromorphic function 
$f$ on $X_{\Sigma}$, we first 
consider the blow-up $\pi_1 \colon 
X_{\Sigma}^{(1)} \longrightarrow X_{\Sigma}$ 
of $X_{\Sigma}$ along the 
$(n-2)$-dimensional smooth subvariety 
$D_1\cap Z$. Then the indeterminacy of 
the pull-back $f \circ \pi_1$ of 
$f$ to $X_{\Sigma}^{(1)}$ is 
improved. If $f \circ \pi_1$ still 
has points of indeterminacy on the 
intersection of the exceptional divisor 
$E_1$ of $\pi_1$ and the proper 
transform $Z^{(1)}$ of $Z$, we construct 
the blow-up $\pi_2 \colon 
X_{\Sigma}^{(2)} \longrightarrow 
X_{\Sigma}^{(1)}$ of $X_{\Sigma}^{(1)}$ 
along $E_1 \cap Z^{(1)}$. By repeating 
this procedure $a_1$ times, we obtain 
a tower of blow-ups
\begin{equation}
X_{\Sigma}^{(a_1)} 
\underset{\pi_{a_1}}{\longrightarrow}
\cdots \cdots
\underset{\pi_2}{\longrightarrow} X_{\Sigma}^{(1)}
\underset{\pi_1}{\longrightarrow} X_{\Sigma}.
\end{equation}
Then the pull-back of $f$ to $X_{\Sigma}^{(a_1)}$ 
has no indeterminacy 
over $T_1$. It also extends to a holomorphic function 
on (an open dense subset of) 
the exceptional divisor of the last blow-up 
$\pi_{a_1}$ whose monodromy at infinity is trivial. 
For this reason we call it 
a horizontal exceptional divisor. For the details 
see the figures in \cite[page 420]{M-T-2}. 
Next we apply this construction 
to the proper transforms of $D_2$ and $Z$ in 
$X_{\Sigma}^{(a_1)}$. Then we obtain 
also a tower of blow-ups
\begin{equation}
X_{\Sigma}^{(a_1)(a_2)} \longrightarrow 
\cdots \cdots \longrightarrow 
X_{\Sigma}^{(a_1)(1)} \longrightarrow 
X_{\Sigma}^{(a_1)}
\end{equation}
and the indeterminacy of the pull-back 
of $f$ to 
$X_{\Sigma}^{(a_1)(a_2)}$ is eliminated 
over $T_1 \sqcup T_2$. By applying 
the same construction to (the proper 
transforms of) $D_3, D_4,\ldots, D_m$, 
we finally obtain a birational morphism 
$\pi \colon \tl{X_{\Sigma}} 
\longrightarrow X_{\Sigma}$ such that 
$g:=f \circ \pi$ has no point of 
indeterminacy on the whole $\tl{X_{\Sigma}}$. 
Note that the smooth 
compactification $\tl{X_{\Sigma}}$ of $\CC^n$ 
thus obtained is not a toric 
variety any more. On $\tl{X_{\Sigma}}$ 
there are $m$ horizontal 
exceptional divisors. By eliminating the 
points of indeterminacy of the 
meromorphic extension of $f$ to $X_{\Sigma}$ 
we have constructed the commutative 
diagram:

\begin{equation}
\begin{CD}
\CC^n  @>{\iota}>> \tl{X_{\Sigma}}
\\
@V{f}VV   @VV{g}V
\\
\CC @>>{j}> \PP^1. 
\end{CD}
\end{equation}
Take a local coordinate $h$ of 
$\PP^1$ in a neighborhood of $\infty 
\in \PP^1$ such that 
$\infty=\{h=0\}$ and set $\tl{g}=h\circ g$, 
$Y=\tl{g}^{-1}(0)=g^{-1}(\infty) 
\subset \tl{X_{\Sigma}}$ and 
$\Omega=\Int(\iota(\CC^n) \sqcup Y)$ as before. 
For simplicity, let us set 
$\tl{g}=\frac{1}{f}$. Then the divisor 
$U=Y \cap \Omega$ in $\Omega$ 
contains not only the proper transforms 
$D_1^{\prime}, \ldots, D_m^{\prime}$ 
of $D_1, \ldots, D_m$ in $\tl{X_{\Sigma}}$ 
but also the exceptional divisors 
of the blow-up: $\tl{X_{\Sigma}} 
\longrightarrow X_{\Sigma}$. 
So the motivic Milnor fiber at infinity 
$\SS_f^{\infty}$ of $f \colon 
\CC^n \longrightarrow \CC$ defined by 
this compactification 
$\tl{X_{\Sigma}}$ of $\CC^n$ 
contains also unramified 
Galois coverings of some subsets of 
these exceptional 
divisors. However they 
are not necessary to compute 
the Hodge realization of $\SS_f^{\infty}$ 
as follows. 
For each non-empty subset 
$I \subset \{1,2,\ldots, m\}$, set 
$D_I= \bigcap_{i \in I} D_i$,
\begin{equation}
D_I^{\circ}=D_I \setminus \left\{ \( 
\bigcup_{i \notin I}D_i\) \cup 
\overline{f^{-1}(0)}\right\} \subset X_{\Sigma}
\end{equation}
and $d_I = {\rm gcd} (a_i)_{i \in I} >0$. Then the 
function $\tl{g}=\frac{1}{f}$ is 
regular on $D_I^{\circ}$ and we can 
decompose it as 
$\frac{1}{f}=\tl{g_{1}}(\tl{g_{2}})^{d_I}$ 
globally on a Zariski open 
neighborhood $W$ of $D_I^{\circ}$ in 
$X_{\Sigma}$, where $\tl{g_1}$ is a 
unit on $W$ and $\tl{g_2} \colon W 
\longrightarrow \CC$ is regular. 
Therefore we can construct an unramified 
Galois covering $\tl{D_I^{\circ}}$ 
of $D_I^{\circ}$ with a natural 
$\mu_{d_I}$-action as in \eqref{eq:6-26}. 
Let $[\tl{D_I^{\circ}}]$ be the element 
of the ring $\M_{\CC}^{\hat{\mu}}$ 
which corresponds to $\tl{D_I^{\circ}}$. 
Then as in \cite[Theorem 4.7]{M-T-4} we obtain the 
following result. 

\begin{theorem}\label{thm:7-7}
Assume that $\dim \Gamma_{\infty}(f)=n$ 
and $f$ is non-degenerate at infinity. 
Then we have 
the equality
\begin{equation}
\chi_h\(\SS_f^{\infty}\)= 
\dsum_{I \neq \emptyset}\chi_h
\( (1-\LL)^{ |I| -1} [\tl{D_I^{\circ}}]\)
\end{equation}
in the Grothendieck group $\KK_0(\HSm)$.
\fin
\end{theorem}
For a face at infinity 
$\gamma \prec 
\Gamma_{\infty}(f)$ of $\Gamma_{\infty}(f)$, 
by using the lattice $M_{\gamma}=\ZZ^n 
\cap \LL(\Delta_{\gamma}) 
\simeq \ZZ^{\dim \gamma+1}$ in 
$\LL(\Delta_{\gamma}) \simeq 
\RR^{\dim \gamma+1}$ 
we set $T_{\Delta_{\gamma}}:=\Spec 
(\CC[M_{\gamma}]) \simeq 
(\CC^*)^{\dim \gamma +1}$. Moreover let 
$\LL(\gamma)$ be the smallest affine 
linear subspace of $\RR^n$ containing $\gamma$ 
and for $v \in M_{\gamma}$ 
define their lattice heights $\height 
(v, \gamma) \in \ZZ$ from 
$\LL(\gamma)$ in $\LL(\Delta_{\gamma})$ 
so that we have $\height (0, 
\gamma)=d_{\gamma}>0$. Then to the group 
homomorphism $M_{\gamma} 
\longrightarrow \CC^*$ defined by $v 
\longmapsto \zeta_{d_{\gamma}}^{\height 
(v, \gamma)}$ we can naturally associate an 
element $\tau_{\gamma} \in 
T_{\Delta_{\gamma}}$. We define a Laurent 
polynomial $g_{\gamma}=\sum_{v \in 
M_{\gamma}}b_v x^v$ on $T_{\Delta_{\gamma}}$ by
\begin{equation}
b_v=\begin{cases}
a_v & (v \in \gamma),\\
-1 & (v=0),\\
\ 0 & (\text{otherwise}),
\end{cases}
\end{equation}
where $f=\sum_{v \in \ZZ^n_+} a_v x^v$. 
Then the Newton polytope $NP(g_{\gamma})$ 
of $g_{\gamma}$ is $\Delta_{\gamma}$, 
$\supp g_{\gamma} 
\subset \{ 0\} \sqcup \gamma$ and the 
hypersurface $Z_{\Delta_{\gamma}}^*=
\{ x \in T_{\Delta_{\gamma}}\ |\ 
g_{\gamma}(x)=0\}$ is non-degenerate 
(see \cite[Section 4]{M-T-4}).  
Since $Z_{\Delta_{\gamma}}^* 
\subset T_{\Delta_{\gamma}}$ 
is invariant by 
the multiplication $l_{\tau_{\gamma}} 
\colon  T_{\Delta_{\gamma}} \simto 
T_{\Delta_{\gamma}}$ by $\tau_{\gamma}$, 
$Z_{\Delta_{\gamma}}^*$ admits an 
action of $\mu_{d_{\gamma}}$. We thus 
obtain an element 
$[Z_{\Delta_{\gamma}}^*]$ of 
$\M_{\CC}^{\hat{\mu}}$. 
For a face at infinity $\gamma 
\prec \Gamma_{\infty}(f)$ let $s_{\gamma} >0$ 
be the dimension of the minimal coordinate subspace 
of $\RR^n$ containing $\gamma$ and set $m_{\gamma}= 
s_{\gamma}-\dim \gamma -1 \geq 0$. 
Finally, for $\lambda \in \CC$ and 
an element $H \in \KK_0(\HSm)$ denote 
by $H_{\lambda} \in \KK_0(\HSm)$ 
the eigenvalue $\lambda$-part of 
$H$. Then by applying the proof of 
\cite[Theorem 5.7 (i)]{M-T-4} to the 
geometric situation in Proposition \ref{ADD-2}, 
we obtain the following result. 

\begin{theorem}\label{thm:7-12}
Assume that $\dim \Gamma_{\infty}(f)=n$ 
and $f$ is non-degenerate at infinity. Then 
for any $\lambda \notin A_f$ we have the equality 
\begin{equation}
[H_f^{\infty}]_{\lambda} 
=\chi_h(\SS_f^{\infty})_{\lambda}=\sum_{\gamma} 
\chi_h((1-\LL)^{m_{\gamma}} \cdot 
[Z_{\Delta_{\gamma}}^*])_{\lambda}
\end{equation}
in $\KK_0(\HSm)$, 
where in the sum $\sum_{\gamma}$ 
the face $\gamma$ of $\Gamma_{\infty}(f)$ 
ranges through the admissible ones at infinity. 
\end{theorem}
\begin{proof}
The proof is similar to that of 
\cite[Theorem 5.7 (i)]{M-T-4}. By Proposition \ref{ADD-2} 
the argument at the end of the proof of 
\cite[Theorem 5.7 (i)]{M-T-4} holds for admissible 
faces at infinity of $\Gamma_{\infty}(f)$. 
But it does not hold for non-admissible ones 
by the presence of horizontal $T$-orbits in 
$X_{\Sigma}$. Hence it suffices to avoid 
atypical eigenvalues $\lambda \in A_f$. 
\end{proof}

\section{Main results}\label{sec:5}

In this section, we consider 
non-convenient polynomials 
$f: \CC^n \longrightarrow \CC$ 
such that $\dim \Gamma_{\infty}(f)=n$. 
For $\lambda \in \CC$ and $j \in \ZZ$ let 
$H^{j}(f^{-1}(R);\CC)_{\lambda} \subset 
H^{j}(f^{-1}(R);\CC)$ be the 
generalized eigenspace for the eigenvalue 
$\lambda$ of the monodromy at infinity 
$\Phi_{j}^{\infty} \colon 
H^{j}(f^{-1}(R);\CC) \simto H^j (f^{-1}(R);\CC)$ 
($R \gg 0$). Denote by $\Phi_{j, \lambda}^{\infty}$ 
the restriction of $\Phi_{j}^{\infty}$ to 
$H^{j}(f^{-1}(R);\CC)_{\lambda}$. Assuming also 
that $f$ is non-degenerate at infinity, for 
non-atypical eigenvalues $\lambda \notin A_f$ 
of $f$ we will prove the concentration 
\begin{equation}
H^{j}(f^{-1}(R);\CC)_{\lambda} \simeq 0 \qquad 
(j \not= n-1) 
\end{equation}
for the $\lambda$-parts $H^{j}(f^{-1}(R);\CC)_{\lambda}$ 
of the cohomology groups of the generic 
fiber $f^{-1}(R)$ ($R \gg 0$) of $f$. 
This implies that the Jordan normal forms of 
the $\lambda$-parts $\Phi_{j, \lambda}^{\infty}$ 
of the monodromies at infinity of $f$ can be 
completely determined by $\Gamma_{\infty}(f)$ 
as in \cite[Section 5]{M-T-4}. For this purpose 
we first consider Laurent polynomials on 
$T=( \CC^*)^n$.  Let $f^{\prime} \in 
\CC [x_1^{\pm 1}, \ldots, x_n^{\pm 1}]$ be a 
Laurent polynomial on 
$T=( \CC^*)^n$.  We define its Newton polytope 
$NP(f^{\prime}) \subset \RR^n$ as usual and let 
$\Gamma_{\infty}(f^{\prime}) \subset \RR^n$ 
be the convex hull of $\{ 0 \} \cup NP(f^{\prime})$ 
in $\RR^n$. We say that a face $\gamma 
\prec \Gamma_{\infty}(f^{\prime})$ 
is at infinity if $0 \notin \gamma$. 
By using faces at infinity of 
$\Gamma_{\infty}(f^{\prime})$ we define also 
the non-degeneracy at infinity 
of $f^{\prime}$ as in Definition \ref{dfn:3-3}. 

\begin{definition} 
Assume that $\dim \Gamma_{\infty}(f^{\prime})=n$. 
Then we say that a face $\gamma \prec 
\Gamma_{\infty}(f^{\prime})$ 
is atypical if $0 \in \gamma$. Moreover a face 
at infinity $\gamma \prec 
\Gamma_{\infty}(f^{\prime})$ is called 
admissible if it is not contained in 
any atypical one. 
\end{definition}
As in Definition \ref{AEV}, by using non-admissible 
faces at infinity of $\Gamma_{\infty}(f^{\prime})$ 
we define the subset $A_{f^{\prime}} \subset 
\CC$ of the atypical eigenvalues of $f^{\prime}$ 
such that $1 \in A_{f^{\prime}}$. 
Finally let us recall the following result of 
Libgober-Sperber \cite{L-S} 
on the monodromies at infinity 
$\Psi_j^{\infty} :  H^{j}( (f^{\prime})^{-1}(R);\CC) 
\simto H^{j}( (f^{\prime})^{-1}(R);\CC)$ ($R \gg 0$) 
of $f^{\prime} : T=(\CC^*)^n \longrightarrow \CC$. 
We define the monodromy zeta 
function at infinity 
$\zeta_{f^{\prime}}^{\infty}(t) \in \CC ((t))$ 
of $f^{\prime}$ by 
\begin{equation}
\zeta_{f^{\prime}}^{\infty}(t)= \prod_{j=0}^{n-1} 
\det ( \id -t \Psi_{j}^{\infty})^{(-1)^j} 
\in \CC ((t)). 
\end{equation}
For a face at infinity 
$\gamma \prec \Gamma_{\infty}(f^{\prime})$ let  
$\LL ( \gamma ) \simeq \RR^{\dim \gamma}$ be 
the minimal affine subspace of $\RR^n$ 
containing $\gamma$. 

\begin{proposition}\label{LST} 
(Libgober-Sperber \cite{L-S}) 
Assume that $\dim \Gamma_{\infty}(f^{\prime})=n$ 
and $f^{\prime}$ is non-degenerate at infinity. 
Then we have 
\begin{equation}
\zeta_{f^{\prime}}^{\infty}(t)= \prod_{\gamma} 
(1 -t^{d_{\gamma}}
)^{(-1)^{n-1} \Vol_{\ZZ}( \gamma )} 
\in \CC ((t)), 
\end{equation}
where in the product $\prod_{\gamma}$ the face 
$\gamma \prec \Gamma_{\infty}(f^{\prime})$ 
ranges through those 
at infinity such that 
$\dim \gamma =n-1$ and 
$\Vol_{\ZZ}( \gamma ) \in \ZZ_{>0}$ 
is the normalized $(n-1)$-dimensional 
volume of $\gamma$ with respect to the 
lattice $\LL ( \gamma ) \cap \ZZ^n \simeq 
\ZZ^{n-1}$. 
\fin
\end{proposition}

\begin{proposition}\label{ABC} 
Let $f^{\prime} \in 
\CC [x_1^{\pm 1}, \ldots, x_n^{\pm 1}]$ be a 
Laurent polynomial on 
$T=( \CC^*)^n$ such that 
$\dim \Gamma_{\infty}(f^{\prime})=n$. Assume 
that $f^{\prime}$ is non-degenerate at infinity. For 
$\lambda \in \CC$ and $j \in \ZZ$ denote the 
generalized eigenspace for the eigenvalue 
$\lambda$ of its monodromy at infinity 
$\Psi_j^{\infty} :  H^{j}( (f^{\prime})^{-1}(R);\CC) 
\simto H^{j}( (f^{\prime})^{-1}(R);\CC)$ ($R \gg 0$) 
by $H^{j}( (f^{\prime})^{-1}(R);\CC)_{\lambda} \subset 
H^{j}( (f^{\prime})^{-1}(R);\CC)$. 
Then for any $\lambda \notin A_{f^{\prime}}$ 
we have the concentration 
\begin{equation}
H^{j}( (f^{\prime})^{-1}(R);\CC)_{\lambda} 
\simeq 0 \qquad (j \not= n-1) 
\end{equation}
for the generic 
fiber $(f^{\prime})^{-1}(R) \subset T$ 
($R \gg 0$) of $f^{\prime}$. If 
$\lambda \notin A_{f^{\prime}}$ 
satisfies the condition 
$H^{n-1}( (f^{\prime})^{-1}(R);\CC)_{\lambda} 
\not= 0$, then there exists a facet at infinity 
$\gamma \prec 
\Gamma_{\infty}(f^{\prime})$ such that 
$\lambda^{d_{\gamma}} = 1$. 
Moreover for such $\lambda$ the relative 
monodromy filtration of 
$H^{n-1}( (f^{\prime})^{-1}(R);\CC)_{\lambda}$ 
($R \gg 0$) coincides 
with the absolute one (up to some shift).
\end{proposition}

\begin{proof}
We will prove the proposition by induction on $n$. 
If $n=1$ the assertion is obvious. Assume that we 
already proved it for the lower dimensions 
$1,2, \ldots, n-1$. Let $\Sigma_1^{\prime}$ 
be the dual fan of $\Gamma_{\infty}(f^{\prime})$ 
in $\RR^n$ and $\Sigma^{\prime}$ its smooth 
subdivision. Then the toric variety 
$X_{\Sigma^{\prime}}$ associated to $\Sigma^{\prime}$ 
is a smooth compactification 
of $T$. By eliminating the points of 
indeterminacy of the meromorphic extension of 
$f^{\prime}$ to $X_{\Sigma^{\prime}}$ 
as in Section \ref{sec:3} we obtain a commutative 
diagram:

\begin{equation}
\begin{CD}
T  @>{\iota^{\prime}}>> \tl{X_{\Sigma^{\prime}}}
\\
@V{f^{\prime}}VV   @VV{g^{\prime}}V
\\
\CC @>>{j}> \PP^1 
\end{CD}
\end{equation}
of holomorphic maps, where $j$ and $\iota^{\prime}$ 
are open embeddings and $g^{\prime}$ is proper. 
Now restricting the map 
$g^{\prime}: \tl{X_{\Sigma^{\prime}}} 
\longrightarrow \PP^1$ 
to $\CC \subset \PP^1$ we set 
$K^{\prime}=(g^{\prime})^{-1}( \CC )= 
\tl{X_{\Sigma^{\prime}}} 
\setminus (g^{\prime})^{-1} (\infty )$. Let 
$\kappa^{\prime} : K^{\prime} \longrightarrow \CC$ 
be the restriction of $g^{\prime}$ to 
$K^{\prime}$. Set $D^{\prime}=K^{\prime} \setminus 
T$ and let 
$i_{D^{\prime}}: D^{\prime} \longrightarrow K^{\prime}$ 
and $i_T : T \longrightarrow K^{\prime}$ 
be the inclusions. Then we obtain also a commutative 
diagram:

\begin{equation}
\begin{CD}
T  @>{i_T}>> K^{\prime} 
\\
@V{f^{\prime}}VV   @VV{\kappa^{\prime}}V
\\
\CC @= \CC . 
\end{CD}
\end{equation}
Note that the normal crossing divisor $D^{\prime}$ 
in $K^{\prime}$ is a union of 
horizontal $T$-orbits 
(which correspond to atypical faces $\gamma \prec 
\Gamma_{\infty}(f^{\prime})$) and 
the horizontal exceptional divisors 
on $\tl{X_{\Sigma^{\prime}}}$. 
By our induction hypothesis 
and Proposition \ref{LST}, 
for $\lambda \notin A_{f^{\prime}}$ 
the monodromies at infinity of the restrictions 
of $\kappa^{\prime}$ to these horizontal 
$T$-orbits have no $\lambda$-part. 
Moreover the corresponding 
monodromies at infinity over the 
horizontal exceptional divisors 
have only the eigenvalue $1 \in  A_{f^{\prime}}$. 
On the other hand, by applying the 
functor $R \kappa^{\prime}_*=R \kappa^{\prime}_!$ 
to the distinguished triangle 
\begin{equation}
(i_T)_! \CC_{T} \longrightarrow 
R (i_T)_* \CC_{T} \longrightarrow 
(i_{D^{\prime}})_* 
i_{D^{\prime}}^{-1} (R (i_T)_* \CC_{T}) 
\longrightarrow +1 
\end{equation}
we obtain a distinguished triangle 
\begin{equation}
R(f^{\prime})_! \CC_{T} \longrightarrow 
R(f^{\prime})_* \CC_{T} \longrightarrow 
R(\kappa^{\prime} |_{D^{\prime}})_* 
i_{D^{\prime}}^{-1} 
(R (i_T)_* \CC_{T}) \longrightarrow +1. 
\end{equation}
Then by using the above description of 
$\kappa^{\prime} |_{D^{\prime}} 
: D^{\prime} \longrightarrow \CC$, 
for $\lambda \notin A_{f^{\prime}}$ we can 
easily show 
\begin{equation}
\psi_{h, \lambda}(j_! 
R(\kappa^{\prime} |_{D^{\prime}})_* 
i_{D^{\prime}}^{-1} (R (i_T)_* \CC_{T}) 
) \simeq 0,
\end{equation}
where $\psi_{h, \lambda}$ is the $\lambda$-part 
of the nearby cycle functor $\psi_h$. 
This implies that there exists 
an isomorphism 
\begin{equation}\label{PQRS} 
\psi_{h, \lambda}(j_! 
R(f^{\prime})_! \CC_{T}
) \simeq \psi_{h, \lambda}(j_! 
R(f^{\prime})_* \CC_{T}). 
\end{equation}
Namely, for any 
$\lambda \notin A_{f^{\prime}}$ and 
$j \in \ZZ$ we have an isomorphism 
\begin{equation}
H^{j}_c( (f^{\prime})^{-1}(R);\CC)_{\lambda} 
\simeq 
H^{j}( (f^{\prime})^{-1}(R);\CC)_{\lambda} 
\qquad (R\gg 0). 
\end{equation}
Since the generic 
fiber $(f^{\prime})^{-1}(R) \subset T$ 
($R \gg 0$) of $f^{\prime}$ is affine, 
the left (resp. right) hand side is zero 
for $j<n-1$ (resp. $j>n-1$). Hence we obtain 
the desired concentration 
\begin{equation}
H^{j}( (f^{\prime})^{-1}(R);\CC)_{\lambda} 
\simeq 0 \qquad (j \not= n-1) 
\end{equation}
for $R \gg 0$. Now the second assertion 
follows immediately from 
Proposition \ref{LST}. Also 
the last assertion 
follows from the proof of Sabbah 
\cite[Theorem 13.1]{Sabbah-2} 
by using the isomorphism \eqref{PQRS}. 
This completes the proof. 
\end{proof}

\begin{remark}
If $0 \in \Int (NP(f^{\prime}))$ then 
the Laurent polynomial 
$f^{\prime}: T=(\CC^*)^n \longrightarrow \CC$ 
is cohomologically tame at infinity in the sense of 
N{\'e}methi-Sabbah \cite{N-S} and Sabbah \cite{Sabbah-2} 
on our compactification $\tl{X_{\Sigma^{\prime}}}$ of $T$. 
In this case the first assertion of Proposition \ref{ABC} 
is due to \cite{N-S}. 
\end{remark}

\begin{theorem}\label{CONC} 
Let $f \in \CC [x_1, \ldots, x_n]$ be a 
non-convenient polynomial such that 
$\dim \Gamma_{\infty}(f)=n$. Assume 
that $f$ is non-degenerate at infinity. 
Then for any non-atypical eigenvalue 
$\lambda \notin A_f$ of $f$ we have 
the concentration 
\begin{equation}
H^{j}( f^{-1}(R);\CC)_{\lambda} 
\simeq 0 \qquad (j \not= n-1) 
\end{equation}
for the generic 
fiber $f^{-1}(R) \subset \CC^n$ 
($R \gg 0$) of $f$. Moreover for 
such $\lambda$ the relative 
monodromy filtration of 
$H^{n-1}( f^{-1}(R);\CC)_{\lambda}$ 
($R \gg 0$) coincides 
with the absolute one (up to some shift). 
\end{theorem}

\begin{proof}
We will freely use the notations 
in Section \ref{sec:3}. For example, 
we consider the commutative 
diagram:

\begin{equation}
\begin{CD}
\CC^n  @>{\iota}>> \tl{X_{\Sigma}}
\\
@V{f}VV   @VV{g}V
\\
\CC @>>{j}> \PP^1. 
\end{CD}
\end{equation}
By restricting the map 
$g: \tl{X_{\Sigma}} \longrightarrow \PP^1$ 
to $\CC \subset \PP^1$ we set 
$K=g^{-1}( \CC )= \tl{X_{\Sigma}} 
\setminus g^{-1} (\infty )$ and 
$\kappa =g|_K: K \longrightarrow \CC$. 
Set $D=K \setminus \CC^n$ and let 
$i_D: D \longrightarrow K$ 
and $i : \CC^n \longrightarrow K$ 
be the inclusions. Then we obtain also 
a commutative diagram:

\begin{equation}
\begin{CD}
\CC^n  @>{i}>> K 
\\
@V{f}VV   @VV{\kappa}V
\\
\CC @= \CC . 
\end{CD}
\end{equation}
Note that the normal crossing divisor $D$ 
in $K$ is a union of horizontal $T$-orbits 
(which correspond to atypical faces 
$\gamma \prec \Gamma_{\infty}(f)$) and 
the horizontal exceptional divisors 
on $\tl{X_{\Sigma}}$. 
By Proposition \ref{ABC}, 
for $\lambda \notin A_f$ 
the monodromies at infinity of the restrictions 
of $\kappa$ to these horizontal $T$-orbits 
have no $\lambda$-part.  
Moreover the corresponding 
monodromies at infinity over the 
horizontal exceptional divisors 
have only the eigenvalue $1 \in  A_{f}$. 
On the other hand, by applying the 
functor $R \kappa_*=R \kappa_!$ 
to the distinguished triangle 
\begin{equation}
 i_! \CC_{\CC^n} \longrightarrow 
R i_* \CC_{\CC^n} \longrightarrow 
(i_D)_* i_D^{-1} (R i_* \CC_{\CC^n}) 
\longrightarrow +1 
\end{equation}
we obtain a distinguished triangle 
\begin{equation}
Rf_! \CC_{\CC^n} \longrightarrow 
Rf_* \CC_{\CC^n} \longrightarrow 
R(\kappa |_D)_* i_D^{-1} 
(R i_* \CC_{\CC^n}) 
\longrightarrow +1. 
\end{equation}
Then by using the above description of 
$\kappa |_D: D \longrightarrow \CC$, 
for $\lambda \notin A_f$ we can 
easily show 
\begin{equation}
\psi_{h, \lambda}(j_! R( \kappa |_D)_* i_D^{-1} 
(R i_* \CC_{\CC^n})) \simeq 0. 
\end{equation}
This implies that there exists an isomorphism 
\begin{equation}\label{PQR} 
\psi_{h, \lambda}(j_! Rf_! \CC_{\CC^n}) \simeq 
\psi_{h, \lambda}(j_! Rf_* \CC_{\CC^n}). 
\end{equation}
Namely, for any 
$\lambda \notin A_f$ and 
$j \in \ZZ$ we have an isomorphism 
\begin{equation}
H^{j}_c( f^{-1}(R);\CC)_{\lambda} 
\simeq 
H^{j}( f^{-1}(R);\CC)_{\lambda} 
\qquad (R\gg 0). 
\end{equation}
Since the generic fiber $f^{-1}(R) \subset \CC^n$ 
($R \gg 0$) of $f$ is affine, 
the left (resp. right) hand side is zero 
for $j<n-1$ (resp. $j>n-1$). Hence we obtain 
the desired concentration 
\begin{equation}
H^{j}( f^{-1}(R);\CC)_{\lambda} 
\simeq 0 \qquad (j \not= n-1) 
\end{equation}
for $R \gg 0$. Moreover the last assertion 
follows from the proof of Sabbah 
\cite[Theorem 13.1]{Sabbah-2} 
by using the isomorphism \eqref{PQR}. 
This completes the proof. 
\end{proof}

\begin{remark}
As is clear from the proof above, Theorem 
\ref{CONC} can be easily generalized to 
arbitrary polynomial maps $f: U \longrightarrow 
\CC$ of affine algebraic varieties $U$. 
We leave the precise formulation to the reader. 
\end{remark}

If $n=2$ the first assertion of Theorem \ref{CONC} 
can be improved as follows. 

\begin{lemma}\label{2DIM} 
Assume that a polynomial $f(x,y) \in \CC [x,y]$ 
of two variables is non-degenerate at infinity and 
satisfies the condition 
$\dim \Gamma_{\infty}(f)=2$. Then the 
generic fiber of $f: \CC^2 
\longrightarrow \CC$ is connected and 
hence $H^0(f^{-1}(R); \CC) \simeq \CC$ 
for $R\gg 0$. In particular, for any 
$\lambda \not= 1$ we have 
the concentration 
\begin{equation}
H^{j}(f^{-1}(R);\CC)_{\lambda} \simeq 0 \qquad 
(j \not= 1) 
\end{equation}
($R \gg 0$). 
\end{lemma}

\begin{proof}
By the classification of 
open connected Riemann surfaces, 
it suffices to show that 
there is no decomposition of $f$ of the form
\begin{equation}
f(x,y)= \hat{f}( \tilde{f}(x,y)) 
\end{equation}
by polynomials $\hat{f}(t)$ and $\tilde{f}(x,y)$ 
such that ${\rm deg} \hat{f}(t) \geq 2$. 
Assume that there exists such a decomposition 
$f=\hat{f} \circ \tilde{f}$ and set 
$m={\rm deg} \hat{f} \geq 2$. Then we have 
$\Gamma_{\infty}(f)=m \Gamma_{\infty}( \tilde{f} 
)$. Take a face at infinity $\gamma \prec 
\Gamma_{\infty}(f)$ of $\Gamma_{\infty}(f)$ 
satisfying $\dim \gamma =1$ and let 
$\tilde{\gamma} \prec \Gamma_{\infty}( \tilde{f} 
)$ be the corresponding one of 
$\Gamma_{\infty}( \tilde{f} )$ such that 
$\gamma =m \tilde{\gamma}$. Denote by $f_{\gamma}$ 
(resp. $\tilde{f}_{\tilde{\gamma}}$) the 
$\gamma$-part of $f$ (resp.  the 
$\tilde{\gamma}$-part of $\tilde{f}$). Then 
we have $f_{\gamma}=(\tilde{f}_{\tilde{\gamma}})^m$ 
(up to some non-zero constant multiple) 
for $m \geq 2$. This contradicts the 
non-degeneracy at infinity of $f$. 
\end{proof}
For an element $[V] \in \KK_0(\HSm)$, 
$V \in \HSm $ with a 
quasi-unipotent endomorphism 
$\Theta \colon V \simto V$, $p, q \geq 0$ and 
$\lambda \in \CC$ denote by 
$e^{p,q}([V])_{\lambda}$ the dimension of the 
$\lambda$-eigenspace of the morphism 
$V^{p,q} \simto V^{p,q}$ induced by 
$\Theta$ on the $(p,q)$-part $V^{p,q}$ of $V$. 
Then by Theorem \ref{CONC} 
we immediately obtain the following result. 

\begin{corollary}
Assume that $\dim \Gamma_{\infty}(f)=n$ 
and $f$ is non-degenerate at infinity. 
Let $\lambda \notin A_f$. 
Then we have $e^{p,q}( 
[H_f^{\infty}])_{\lambda}=0$ 
for $(p,q) \notin [0,n-1] \times [0,n-1]$. 
Moreover for any $(p,q) \in [0,n-1] 
\times [0,n-1]$ we have 
the Hodge symmetry 
\begin{equation}
e^{p,q}( [H_f^{\infty}])_{\lambda}=e^{n-1-q,n-1-p}( 
[H_f^{\infty}])_{\lambda}.
\end{equation} \fin
\end{corollary}

By using the notations in Section \ref{sec:3} 
we thus obtain the following theorem, 
whose proof is similar to that of 
\cite[Theorem 5.7 (ii)]{M-T-4}. 

\begin{theorem}\label{MAIN}
Assume that $\dim \Gamma_{\infty}(f)=n$ 
and $f$ is non-degenerate at infinity. 
Let $\lambda \notin A_f$ and $k \geq 1$. 
Then the number of the Jordan blocks for the 
eigenvalue $\lambda$ with sizes $\geq k$ in 
$\Phi_{n-1}^{\infty} \colon 
H^{n-1}(f^{-1}(R) ;\CC) \simto 
H^{n-1}(f^{-1}(R) ;\CC)$ ($R \gg 0$) is equal to
\begin{equation}
(-1)^{n-1}\sum_{p+q=n-2+k, n-1+k}
\left\{ \sum_{\gamma} 
e^{p,q} ( \chi_h ((1-\LL)^{m_{\gamma}} \cdot 
[Z_{\Delta_{\gamma}}^*] ))_{\lambda} \right\}, 
\end{equation}
where in the sum $\sum_{\gamma}$ 
the face $\gamma$ of $\Gamma_{\infty}(f)$ 
ranges through the admissible ones at infinity. 
\fin
\end{theorem}

By this theorem and the results in 
\cite[Section 2]{M-T-4} we 
immediately obtain the generalizations of 
\cite[Theorems 5.9, 5.14 and 5.16]{M-T-4} 
to non-tame polynomials. 
Here we introduce only that of 
\cite[Theorem 5.9]{M-T-4}. 
Denote by $\Cone_{\infty}(f)$ 
the closed cone $\RR_+ \Gamma_{\infty}(f) 
\subset \RR^n_+$ generated by 
$\Gamma_{\infty}(f)$. 
Let $q_1,\ldots,q_l$ (resp. $\gamma_1,\ldots, 
\gamma_{l^{\prime}}$) be the 
$0$-dimensional (resp. $1$-dimensional) 
faces at infinity of $\Gamma_{\infty}(f)$ such 
that $q_i\in \Int (\Cone_{\infty}(f))$ (resp. 
the relative interior 
$\relint(\gamma_i)$ of $\gamma_i$ is 
contained in 
$\Int(\Cone_{\infty}(f))$). 
For each $q_i$ (resp. $\gamma_i$), 
denote by $d_i >0$ (resp. $e_i>0$) its 
lattice distance from 
the origin $0\in \RR^n$. For $1\leq i 
\leq l^{\prime}$, let $\Delta_i$ be the 
convex hull of $\{0\}\sqcup 
\gamma_i$ in $\RR^n$. Then for $\lambda 
 \not= 1$ and $1 \leq 
i \leq l^{\prime}$ such that $\lambda^{e_i}=1$ we set
\begin{equation}
n(\lambda)_i
= \sharp\{ v\in \ZZ^n \cap \relint(\Delta_i) 
\ |\ \height (v, \gamma_i)=k\} 
+\sharp \{ v\in \ZZ^n \cap \relint(\Delta_i) 
\ |\ \height (v, 
\gamma_i)=e_i-k\},
\end{equation}
where $k$ is the minimal positive 
integer satisfying 
$\lambda=\zeta_{e_i}^{k}$ and for 
$v\in \ZZ^n \cap \relint(\Delta_i)$ we 
denote by $\height (v, \gamma_i)$ the 
lattice height of $v$ from the base 
$\gamma_i$ of $\Delta_i$. 
Then we have the following generalization of 
\cite[Theorem 5.9]{M-T-4}. 

\begin{theorem}
Assume that $\dim \Gamma_{\infty}(f)=n$ 
and $f$ is non-degenerate at infinity. 
Let $\lambda \notin A_f$. Then we have
\begin{enumerate}
\item The number of the Jordan blocks 
for the eigenvalue $\lambda$ with the 
maximal possible size $n$ in 
$\Phi_{n-1}^{\infty} \colon 
H^{n-1}(f^{-1}(R);\CC) \simto H^{n-1}(f^{-1}(R);\CC)$ 
($R \gg 0$) is equal 
to $\sharp \{q_i \ |\ \lambda^{d_i}=1\}$.
\item The number of the Jordan blocks for 
the eigenvalue $\lambda$ with the second maximal possible size 
$n-1$ in $\Phi_{n-1}^{\infty}$ is 
equal to $\sum_{i \colon \lambda^{e_i}=1} 
n(\lambda)_i$.
\end{enumerate} \fin
\end{theorem}

\begin{remark}
By Proposition \ref{ABC} we can 
similarly obtain the analogues of 
\cite[Theorems 5.9, 5.14 and 5.16]{M-T-4} 
for Laurent polynomials $f^{\prime} \in 
\CC [x_1^{\pm 1}, \ldots, x_n^{\pm 1}]$. 
The results on the Jordan normal forms of 
their monodromies at infinity 
for the eigenvalues $\lambda \notin A_{f^{\prime}}$ 
are explicitly described by 
the admissible faces at infinity 
of $\Gamma_{\infty}(f^{\prime})$. 
We omit the details. 
\end{remark}

Moreover in the situation above, we can obtain also 
a closed formula for the multiplicities of the 
non-atypical eigenvalues $\lambda \notin A_f$ in 
the monodromy at infinity 
$\Phi_{n-1}^{\infty} \colon 
H^{n-1}(f^{-1}(R);\CC) \simto H^{n-1}(f^{-1}(R);\CC)$ 
($R \gg 0$) as follows. We define the monodromy zeta 
function at infinity 
$\zeta_f^{\infty}(t) \in \CC ((t))$ of $f$ by 
\begin{equation}
\zeta_f^{\infty}(t)= \prod_{j=0}^{n-1} 
\det ( \id -t \Phi_{j}^{\infty})^{(-1)^j} 
\in \CC ((t)). 
\end{equation}
Then by our compactification 
$\tl{X_{\Sigma}}$ of $\CC^n$ we obtain the 
following refinement of the previous results 
in \cite{L-S} and \cite{M-T-2}. In particular 
here we can remove the condition $( \ast )$ in 
\cite{M-T-2}. 

\begin{theorem}\label{MZF} 
Assume that $\dim \Gamma_{\infty}(f)=n$ 
and $f$ is non-degenerate at infinity. 
Then we have 
\begin{equation}
\zeta_f^{\infty}(t)= \prod_{\gamma} 
(1 -t^{d_{\gamma}}
)^{(-1)^{s_{\gamma}-1} \Vol_{\ZZ}( \gamma )} 
\in \CC ((t)), 
\end{equation}
where in the product $\prod_{\gamma}$ the face 
$\gamma \prec \Gamma_{\infty}(f)$ 
ranges through those 
at infinity satisfying the condition 
$m_{\gamma}=s_{\gamma}- \dim \gamma -1=0$ and 
$\Vol_{\ZZ}( \gamma ) \in \ZZ_{>0}$ 
is the normalized $(\dim \gamma)$-dimensional 
volume of $\gamma$ with respect to the 
lattice $\LL ( \gamma ) \cap \ZZ^n \simeq 
\ZZ^{\dim \gamma}$. 
\fin
\end{theorem}

\begin{example}
Let $n=3$ and consider a non-convenient polynomial 
$f(x,y,z)$ on $\CC^3$ whose Newton polyhedron at 
infinity $\Gamma_{\infty}(f)$ is the convex hull of 
the points $(2,0,0), (0,2,0), (1,1,1) \in \RR^3_+$ 
and the origin $0=(0,0,0) \in \RR^3$. Then the line 
segment connecting the point $(2,0,0)$ 
and the origin $0 \in \RR^3$ is an atypical 
face of $\Gamma_{\infty}(f)$. Hence the $0$-dimensional 
face at infinity $\gamma = \{ (2,0,0) \} \prec 
\Gamma_{\infty}(f)$ of $\Gamma_{\infty}(f)$ 
contained in it is not admissible. However it 
satisfies the condition 
$m_{\gamma}=s_{\gamma}- \dim \gamma -1= 
1-0-1=0$. For the proof of Theorem \ref{MZF} 
we have to consider also the contribution from 
such non-admissible faces at infinity of 
$\Gamma_{\infty}(f)$. 
\end{example}
If we restrict ourselves to the non-atypical 
eigenvalues $\lambda \notin A_f$ for which we 
have the concentration 
\begin{equation}
H^{j}( f^{-1}(R);\CC)_{\lambda} 
\simeq 0 \qquad (j \not= n-1) 
\end{equation}
($R \gg 0$) in Theorem \ref{CONC}, 
we have the following result. 

\begin{corollary}
Assume that $\dim \Gamma_{\infty}(f)=n$ 
and $f$ is non-degenerate at infinity. 
Then for any $\lambda \notin A_f$ 
the multiplicity of the eigenvalue $\lambda$ in 
the monodromy at infinity 
$\Phi_{n-1}^{\infty}$ of $f$ is equal to that 
of the factor $(1- \lambda t)= \lambda \cdot 
(1/ \lambda -t)$ in the rational function 
\begin{equation}
\prod_{\gamma} 
(1 -t^{d_{\gamma}})^{(-1)^{n- s_{\gamma}} 
\Vol_{\ZZ}( \gamma )} 
\in \CC ((t)), 
\end{equation}
where in the product $\prod_{\gamma}$ the face 
$\gamma \prec \Gamma_{\infty}(f)$ of 
$\Gamma_{\infty}(f)$ ranges through the admissible 
ones at infinity satisfying the condition 
$m_{\gamma}=s_{\gamma}- \dim \gamma -1=0$. 
\fin
\end{corollary}

\section{Monodromies around atypical fibers}\label{sec:6}

Let $f: U \longrightarrow \CC$ be a polynomial map 
of an affine algebraic variety $U$ 
and $B_f \subset \CC$ the set of its bifurcation 
points. For a point $b \in B_f$ we choose sufficiently 
small $\e >0$ such that 
\begin{equation}
B_f \cap \{ x \in \CC \ | \ |x-b| \leq \e \} 
= \{ b \} 
\end{equation}
and set $C_{\e}(b)= 
\{ x \in \CC \ | \ |x-b|= \e \} \subset \CC$. 
Then we obtain a locally trivial fibration 
$f^{-1}(C_{\e}(b)) \longrightarrow C_{\e}(b)$ 
over the small circle $C_{\e}(b) \subset \CC$ 
and the monodromy automorphisms 
\begin{equation}
\Phi_j^b \colon H^j(f^{-1}(b+ \e ) ;\CC) 
\simto 
H^j(f^{-1}(b+ \e ) ;\CC) \ \ (j=0,1,\ldots)
\end{equation}
around the atypical fiber $f^{-1}(b) \subset U$ 
associated to it. We can construct $\Phi_j^b$'s 
functorially as follows. 
Let $h^b$ be a holomorphic 
local coordinate of $\CC$ on a neighborhood 
of $b \in B_f$ such that $b= \{ h^b(x)=0 \}$. 
Then to the object $\psi_{h^b}(Rf_!\CC_{U})\in 
\Dbc(\{ b \})$ and the 
semisimple part of the monodromy 
automorphism acting on it, we can associate 
an element
\begin{equation}
[H_f^{b}] \in \KK_0(\HSm). 
\end{equation}
Recall that the weight filtration 
of $[H_f^{b}]$ is a relative one.  In this situation, we can 
apply our methods in previous sections to 
the Jordan normal forms of $\Phi_j^b$. 
For the sake of simplicity, let us assume 
here that the central fiber $f^{-1}(b) \subset U$ 
is reduced and has only isolated singular 
points $p_1, p_2, \ldots, p_l 
\in f^{-1}(b) \subset U$. 
When $f$ is not tame at infinity, we have to 
consider also the singularities at infinity of $f$. 
For this purpose, let $X$ be a smooth 
compactification of $U$ for which 
there exists a commutative diagram 

\begin{equation}
\begin{CD}
U  @>{\iota}>> X 
\\
@V{f}VV   @VV{g}V
\\
\CC @>>{j}> \PP^1
\end{CD}
\end{equation}
of holomorphic maps. Here $\iota$ and $j$ 
are inclusion maps and $g$ is proper. 
We may assume also that the divisor at 
infinity $D=X \setminus U \subset X$ is 
normal crossing and all its irreducible 
components are smooth. We call the irreducible 
components of $D$ contained 
in $g^{-1}( \infty ) \subset D$ (resp. 
in $\overline{D \setminus g^{-1}( \infty )} 
\subset D$) ``vertical" (resp. ``horizontal") 
divisors at infinity of $f$ in $X$. For 
the normal crossing divisor $D$ let us 
consider the standard (minimal) 
stratification. Then for simplicity 
we assume also that 
the restriction 
$g|_{D \setminus g^{-1}( \infty )}: 
D \setminus g^{-1}( \infty ) \longrightarrow 
\CC$ of $g$ to the horizontal part 
$D \setminus g^{-1}( \infty )$ of $D$ 
has only stratified isolated singular 
points $p_{l+1}, \ldots, p_{l+r}$ in 
$g^{-1}(b) \subset X$ and all of them 
are contained in the smooth part 
of $D \setminus g^{-1}( \infty )$. 
By our assumption on $f^{-1}(b) \subset U$ 
this implies that the hypersurface 
$\overline{f^{-1}(b)} =g^{-1}(b) 
\subset X$ in $X$ has also 
an isolated singular point at each 
$p_i$ ($l+1 \leq i \leq l+r$). 

\begin{remark} 
If $U= \CC^n$, $b \not= f(0)$ and 
in addition to the conditions in Theorems 
\ref{CONC} and \ref{MAIN} (i.e. 
$\dim \Gamma_{\infty}(f)=n$ 
and $f$ is non-degenerate at infinity) 
we assume that for any atypical face 
$\gamma \prec \Gamma_{\infty}(f)$ 
such that $\dim \gamma 
<n-1$ the $\gamma$-part $f_{\gamma}: 
(\CC^*)^{n} \longrightarrow \CC$ 
of $f$ does not have the critical value $b$, then 
the meromorphic extension $g$ of $f$ to 
the compactification $X= \tl{X_{\Sigma}}$ 
satisfies the above-mentioned property 
in general (see also N{\'e}methi-Zaharia \cite{N-Z}, 
Zaharia \cite{Zaharia}). 
In this case the stratified 
isolated singular points 
$p_{l+1}, \ldots, p_{l+r}$ are on the 
$(n-1)$-dimensional horizontal $T$-orbits 
which correspond to the atypical facets 
of $\Gamma_{\infty}(f)$. 
\end{remark} 
For $l+1 \leq i \leq l+r$, in a neighborhood of 
$p_i$ the divisor $D$ is smooth and the 
function $g|_D : D \simeq \CC^{n-1} 
\longrightarrow \CC$ 
has an isolated singular point at $p_i \in D$. 
Therefore we may consider the (local) 
Milnor monodromies of 
$g|_D: D \simeq \CC^{n-1} 
\longrightarrow \CC$ at $p_i \in D$. 
Denote by $A_{f,b} \subset \CC$ the union of 
their eigenvalues and $1 \in \CC$. Then by 
applying the proof of Theorem \ref{CONC} 
to this situation, we obtain the following 
result. For $\lambda \in \CC$ and $j \in \ZZ$ let 
$H^{j}(f^{-1}(b + \e );\CC)_{\lambda} \subset 
H^{j}(f^{-1}(b + \e );\CC)$ be the 
generalized eigenspace for the eigenvalue 
$\lambda$ of the monodromy $\Phi_{j}^b$ 
around $f^{-1}(b)$. 

\begin{theorem}\label{CONCE} 
In the situation as above, for any 
$\lambda \notin A_{f,b}$ we have 
the concentration 
\begin{equation}
H^{j}( f^{-1}(b + \e );\CC)_{\lambda} 
\simeq 0 \qquad (j \not= n-1).  
\end{equation}
Moreover for such $\lambda$ the relative 
monodromy filtration of 
$H^{n-1}( f^{-1}(b + \e );\CC)_{\lambda}$ 
coincides with the absolute 
one (up to some shift). 
\fin
\end{theorem}

\begin{corollary}\label{SYMM} 
Let $\lambda \notin A_{f,b}$. 
Then we have $e^{p,q}( 
[H_f^{b}])_{\lambda}=0$ 
for $(p,q) \notin [0,n-1] \times [0,n-1]$. 
Moreover for any $(p,q) \in [0,n-1] 
\times [0,n-1]$ we have 
the Hodge symmetry 
\begin{equation}
e^{p,q}( [H_f^{b}])_{\lambda}=e^{n-1-q,n-1-p}( 
[H_f^{b}])_{\lambda}.
\end{equation} \fin
\end{corollary}
From now on we shall use Theorem \ref{CONCE} 
and Corollary \ref{SYMM} to describe 
explicitly the Jordan normal form 
of $\Phi_{n-1}^b$ in terms of some Newton 
polyhedra associated to $f$. For this purpose, 
assume moreover that for any $1 \leq i \leq l+r$ 
there exists a local coordinate 
$y=(y_1, y_2, \ldots, y_n)$ of $X$ on a 
neighborhood $W_i$ of $p_i$ such that 
$p_i= \{ y=0 \}$ and the 
local defining polynomial $f_i(y) \in 
\CC [y_1, \ldots, y_n]$ of the hypersurface 
$\overline{f^{-1}(b)} =g^{-1}(b)$ (for which we have 
$\overline{f^{-1}(b)} = \{ f_i(y)=0 \}$) 
is convenient and non-degenerate at 
$y=0$ (see \cite{Varchenko} etc.). 
We assume also that for $l+1 \leq i \leq l+r$ 
we have $D= \{ y_n=0 \}$ in $W_i$. 
For $1 \leq i \leq l+r$ let 
$\Gamma_+(f_i) \subset \RR^n_+$ be the 
Newton polyhedron of $f_i$ at $y=0$. 
Moreover for $l+1 \leq i \leq l+r$ we set 
\begin{equation}
\Gamma_+^{\circ}(f_i)= \Gamma_+(f_i) \cap 
\left\{ v=(v_1, \ldots, v_n) \in \RR^n \ | \ 
v_n=0  \right\}. 
\end{equation}
Note that $\Gamma_+^{\circ}(f_i)$ is nothing 
but the Newton polyhedron of the restriction 
$f_i|_D$ of $f_i$ to $D= \{ y_n=0 \}$. 

\begin{definition} 
In the situation as above, 
we say that a complex number $\lambda \in \CC$ 
is an atypical eigenvalue for $b \in B_f$ if 
either $\lambda =1$ or there 
exists a compact face $\gamma 
\prec \Gamma_+^{\circ}(f_i)$ of 
$\Gamma_+^{\circ}(f_i)$ for some 
$l+1 \leq i \leq l+r$ such that 
$\lambda^{d_{\gamma}}=1$. We denote by $A_{f,b}^{\circ} 
\subset \CC$ the set of the atypical 
eigenvalues for $b \in B_f$. 
\end{definition}
By the main theorem of Varchenko 
\cite{Varchenko} we have $A_{f,b} \subset 
A_{f,b}^{\circ}$. On the other hand, 
as in \cite{D-L-2}, \cite{M-T-4} 
and \cite{M-T-5}, for $1 \leq i \leq l+r$ 
by a toric modification 
$\pi_i : Y_i \longrightarrow W_i$ of $W_i$ 
we can explicitly 
construct the motivic Milnor fiber 
$\SS_{f_i,p_i} \in \M_{\CC}^{\hat{\mu}}$ 
of $f_i$ at $p_i$. See \cite{M-T-5} for 
the details. For $l+1 \leq i \leq l+r$ 
let $(W_i \cap D)^{\prime} \subset Y_i$ 
be the proper transform of $W_i \cap D= \{ y_n=0 \}$ 
by $\pi_i$ and $\SS_{f_i,p_i}^{\circ} 
\in \M_{\CC}^{\hat{\mu}}$ the base change of 
$\SS_{f_i,p_i}$ by the inclusion map 
$Y_i \setminus (W_i \cap D)^{\prime} 
\hookrightarrow Y_i$. Let 
$[Z_{f,b}] \in \M_{\CC}^{\hat{\mu}}$ 
be the class of the variety 
$Z_{f,b}= f^{-1}(b) \setminus \{ 
p_1, p_2, \ldots, p_l \}$ with the trivial action of 
$\hat{\mu}$ and set 
\begin{equation}
\SS_f^{b} =[Z_{f,b}] + \sum_{i=1}^l \SS_{f_i,p_i} 
+ \sum_{i=l+1}^{l+r} \SS_{f_i,p_i}^{\circ} 
\in \M_{\CC}^{\hat{\mu}}. 
\end{equation}
Then as in \cite[Theorem 4.4]{M-T-4}, by 
the proof of \cite[Theorem 4.2.1]{D-L-1} 
we obtain the following result.

\begin{theorem}\label{M-M-F}
In $\KK_0(\HSm)$ we have the equality 
\begin{equation}
[H_f^{b}] = \chi_h(\SS_f^{b}). 
\end{equation}\fin
\end{theorem}
By Theorems \ref{CONCE} and \ref{M-M-F} 
and Corollary \ref{SYMM}, for any 
$\lambda \notin A_{f,b}^{\circ}$ we can describe 
explicitly the $\lambda$-part of the 
Jordan normal form 
of $\Phi_{n-1}^b$ as follows. 
For $1 \leq i \leq l+r$ let $\gamma \prec 
\Gamma_+(f_i)$ be a compact face of 
$\Gamma_+(f_i)$. Denote by $\Delta_{\gamma}$ 
the convex hull of $\{0\} 
\sqcup \gamma$ in $\RR^n$. Let 
$\LL(\Delta_{\gamma})$ be the 
$(\dim \gamma +1)$-dimensional linear subspace 
of $\RR^n$ spanned by $\Delta_{\gamma}$ 
and consider the lattice 
$M_{\gamma}=\ZZ^n \cap \LL(\Delta_{\gamma}) 
\simeq \ZZ^{\dim \gamma+1}$ in 
it. Then we set $T_{\Delta_{\gamma}}:
=\Spec (\CC[M_{\gamma}]) \simeq 
(\CC^*)^{\dim \gamma +1}$. Moreover 
for the points $v \in M_{\gamma}$ we 
define their lattice heights $\height 
(v, \gamma) \in \ZZ$ from the affine hyperplane 
$\LL(\gamma)$ in $\LL(\Delta_{\gamma})$ 
so that we have $\height (0, 
\gamma)=d_{\gamma}>0$. 
Then to the group homomorphism $M_{\gamma} 
\longrightarrow \CC^*$ defined by 
$v \longmapsto 
\zeta_{d_{\gamma}}^{-\height (v, \gamma)}$ 
we can naturally associate an 
element $\tau_{\gamma} \in 
T_{\Delta_{\gamma}}$. 
We define a Laurent polynomial 
$g_{\gamma}=\sum_{v \in 
M_{\gamma}}b_v y^v$ 
on $T_{\Delta_{\gamma}}$ by
\begin{equation}
b_v=\begin{cases}
a_v & (v \in \gamma),\\
-1 & (v=0),\\
\ 0 & (\text{otherwise}),
\end{cases}
\end{equation}
where $f_i=\sum_{v \in \ZZ^n_+} a_v y^v$. 
Then we have $NP(g_{\gamma}) 
=\Delta_{\gamma}$, $\supp g_{\gamma} 
\subset \{ 0\} \sqcup \gamma$ and 
the hypersurface $Z_{\Delta_{\gamma}}^*
=\{ y \in T_{\Delta_{\gamma}}\ |\ 
g_{\gamma}(y)=0\}$ is non-degenerate 
by \cite[Proposition 5.3]{M-T-4}. 
Moreover $Z_{\Delta_{\gamma}}^* \subset 
T_{\Delta_{\gamma}}$ is invariant by the 
multiplication $l_{\tau_{\gamma}} 
\colon T_{\Delta_{\gamma}} \simto 
T_{\Delta_{\gamma}}$ by 
$\tau_{\gamma}$, and hence we obtain an element 
$[Z_{\Delta_{\gamma}}^*]$ of 
$\M_{\CC}^{\hat{\mu}}$. Finally we 
define $m_{\gamma} \in \ZZ_+$ as in 
Section \ref{sec:3}. 
Then in the same way as \cite[Theorem 5.7]{M-T-4} and 
\cite[Theorem 4.3]{M-T-5} we 
obtain the following results. 

\begin{theorem}\label{MMFS}
In the situation as above, 
for any $\lambda \notin A_{f,b}^{\circ}$ 
we have the equality 
\begin{equation}
[H_f^{b}]_{\lambda} =
\chi_h(\SS_f^{b})_{\lambda}= 
\sum_{i=1}^{l+r} 
\sum_{\gamma \prec \Gamma_+(f_i)} 
\chi_h((1-\LL)^{m_{\gamma}} \cdot 
[Z_{\Delta_{\gamma}}^*])_{\lambda}
\end{equation}
in $\KK_0(\HSm)$, 
where in the sum 
$\sum_{\gamma \prec \Gamma_+(f_i)}$ 
for $l+1 \leq i \leq l+r$ 
the face $\gamma$ of $\Gamma_{+}(f_i)$ 
ranges through compact ones not contained 
in $\Gamma_+^{\circ}(f_i)$. 
\fin
\end{theorem}

\begin{theorem}\label{MAINS}
In the situation as above, let 
$\lambda \notin A_{f,b}^{\circ}$ and $k \geq 1$. 
Then the number of the Jordan blocks for the 
eigenvalue $\lambda$ with sizes $\geq k$ in 
$\Phi_{n-1}^{b}$ is equal to
\begin{equation}
(-1)^{n-1}\sum_{p+q=n-2+k, n-1+k}
\left\{ \sum_{i=1}^{l+r} 
\sum_{\gamma \prec \Gamma_+(f_i)} 
e^{p,q} ( \chi_h ((1-\LL)^{m_{\gamma}} \cdot 
[Z_{\Delta_{\gamma}}^*] ))_{\lambda} \right\}, 
\end{equation}
where in the sum 
$\sum_{\gamma \prec \Gamma_+(f_i)}$ 
for $l+1 \leq i \leq l+r$ 
the face $\gamma$ of $\Gamma_{+}(f_i)$ 
ranges through compact ones not contained 
in $\Gamma_+^{\circ}(f_i)$. 
\fin
\end{theorem}

By this theorem and the results in 
\cite[Section 2]{M-T-4}, 
for $\lambda \notin A_{f,b}^{\circ}$ 
we immediately obtain the analogues of 
\cite[Theorems 5.9, 5.14 and 5.16]{M-T-4} 
for the $\lambda$-part of the 
Jordan normal form of $\Phi_{n-1}^{b}$. 
More precisely it suffices to 
neglect the compact faces of 
$\Gamma_+^{\circ}(f_i)$ for $l+1 \leq i \leq l+r$. 
We omit the details.

\end{document}